\documentclass[11pt,a4paper]{article}

\usepackage[utf8]{inputenc}  
\usepackage[T1]{fontenc}  
\usepackage[english]{babel}
\usepackage{tikz}
\usepackage{cite}
\usetikzlibrary{positioning,chains,fit,shapes,calc}

\usepackage{geometry}
\usepackage{todonotes}

\usepackage{amsmath}
\usepackage{amsfonts}
\usepackage{amssymb}
\usepackage{amsthm}
\usepackage{stmaryrd}
\usepackage{dsfont}
\usepackage{bbm}

\newtheorem{theorem}{Theorem}
\newtheorem{lemma}[theorem]{Lemma}

\newtheorem{proposition}[theorem]{Proposition}
\newtheorem{cor}[theorem]{Corollary}

\newtheorem{remark}{Remark}

\newtheorem*{theorem*}{Theorem}
\newtheorem*{proposition*}{Proposition}
\newtheorem*{lemma*}{Lemma}
\usepackage{graphicx}
\usepackage{float}
\restylefloat{figure}

\newcommand{\E}{\mathbb{E}}
\renewcommand{\P}{\mathbb{P}}

\newcommand{\R}{\mathbb{R}}

\newcommand{\Z}{\mathbb{Z}}
\newcommand{\C}{\mathbb{C}}

\newcommand{\kT}{\mathfrak{T}}
\newcommand{\kD}{\mathfrak{D}}
\newcommand{\kS}{\mathfrak{S}}
\newcommand{\cT}{\mathcal{T}}

\newcommand{\cA}{\mathcal{A}}
\newcommand{\cC}{\mathcal{C}}

\DeclareMathOperator*{\argmax}{arg\,max}

\DeclareMathOperator*{\rd}{rd}
\newcommand{\PL}{\ell}
\newcommand{\Bin}{\operatorname{Bin}}
\newcommand{\Exp}{\operatorname{Exp}}
\newcommand{\Geom}{\operatorname{Geom}}
\newcommand{\Poi}{\operatorname{Poisson}}

\usepackage[colorlinks=false]{hyperref}

\author{Paul Thévenin, Stephan Wagner}
\title{Local limits of descent-biased permutations and trees}
\date{}

\AtEndDocument{%
  \small{\begin{tabular}{@{}l@{}}%
    \textsc{Paul Thévenin, Institut für Mathematik, University of Vienna, Vienna, Austria}\\
    \textsc{and Department of Mathematics, Uppsala University, Uppsala, Sweden}\\
    \textit{Email address}:
    \texttt{paul.thevenin@univie.ac.at}\\
    \textsc{Stephan Wagner, Department of Mathematics, Uppsala University, Uppsala, Sweden}\\
    \textsc{and Institute of Discrete Mathematics, TU Graz, Graz, Austria}\\
    \textit{Email address}:
    \texttt{stephan.wagner@math.uu.se}
  \end{tabular}}}

\begin{document}

\maketitle

\begin{abstract}
\small{We study two related probabilistic models of permutations and trees biased by their number of descents. Here, a descent in a permutation $\sigma$ is a pair of consecutive elements $\sigma(i), \sigma(i+1)$ such that $\sigma(i) > \sigma(i+1)$. Likewise, a descent in a rooted tree with labelled vertices is a pair of a parent vertex and a child such that the label of the parent is greater than the label of the child. For some nonnegative real number $q$, we consider the probability measures on permutations and on rooted labelled trees of a given size where each permutation or tree is chosen with a probability proportional to $q^{\text{number of descents}}$. In particular, we determine the asymptotic distribution of the first elements of permutations under this model. Different phases can be observed based on how $q$ depends on the number of elements $n$ in our permutations. The results on permutations then allow us to characterize the local limit of descent-biased rooted labelled trees.}
\end{abstract}

\section{Introduction}

For an integer $n \geq 1$, let $\kS_n$ be the set of permutations of $\llbracket 1, n \rrbracket := \{1,2,\ldots,n\}$. We consider random permutations $\sigma \in \kS_n$ biased by their number of descents, that is, the number of integers $i \in \llbracket 1, n-1 \rrbracket$ such that $\sigma(i)> \sigma(i+1)$. More precisely, for any $q \in [0,\infty)$, we consider the random variable $S_n^{(q)}$ on $\kS_n$ such that, for any permutation $\sigma \in \kS_n$,
\begin{align*}
\P\left( S_n^{(q)}=\sigma \right) = \frac{1}{Z_{n,q}} q^{d(\sigma)}, 
\end{align*}
where $d(\sigma)$ is the number of descents of $\sigma$ and $Z_{n,q}=\sum_{\sigma \in \kS_n} q^{d(\sigma)}$ (with the convention $0^0=1$). In particular, when $q=1$, we recover the uniform distribution on $\kS_n$. On the other hand, when $q=0$, $S_n^{(0)}$ is almost surely the identity permutation. Observe also that, by symmetry, for any $q>0$, the behaviour of $S_n^{(q)}$ is essentially the same as the behaviour of $S_n^{(1/q)}$, in the sense that $(S_n^{(q)}(k))_{1 \leq k \leq n}$ has the same distribution as $(n+1-S_n^{(1/q)}(k))_{1 \leq k \leq n}$. Hence, we can restrict ourselves to the case where $q \in [0,1]$.

This probabilistic model is in some sense a local variant of the more frequently studied model of Mallows permutations introduced in \cite{Mal57}. The $q$-Mallows permutation of size $n$ is the random permutation $V_n^{(q)}$ satisfying, for any permutation $\tau \in \kS_n$:
\begin{align*}
\P\left( V_n^{(q)}=\tau\right) = \frac{1}{\tilde{Z}_{n,q}} q^{i(\tau)}, 
\end{align*}
where $i(\tau)$ is the number of inversions of $\tau$, that is, the number of pairs $(j,k) \in \llbracket 1, n \rrbracket^2$ such that $j<k$ and $\tau(j)>\tau(k)$, and $\tilde{Z}_{n,q}=\sum_{\tau \in \kS_n} q^{i(\tau)}$. Scaling limit results for this model have been shown in~\cite{Sta09}: after rescaling, the empirical measure of a $q_n$-Mallows permutation of size $n$, in the regime where $n(1-q_n)$ converges to a positive constant, converges to a certain random measure on $[0,1]^2$.

Our first goal in this paper is to prove a local limit theorem on $S_n^{(q_n)}$, where $(q_n)_{n \geq 1}$ is a given sequence of real numbers in $[0,1]$, by studying the asymptotic behaviour of $(S_n^{(q_n)}(1), \ldots, S_n^{(q_n)}(r))$ for fixed $r \geq 1$. We distinguish three phases: a degenerate one, where $q_n$ decreases at least exponentially with $n$, a nondegenerate one where $q_n$ goes to $0$ more slowly, and a constant phase where $q_n$ remains constant as $n \to \infty$.

\paragraph*{The local limit of descent-biased permutations.}

We remark that the case where $q_n$ goes to $0$ faster than exponentially is trivial, as we then have for all fixed $r \geq 1$ that $\left(S_n^{(q_n)}(1), \ldots, S_n^{(q_n)}(r)\right) = (1,2,\ldots,r)$ with high probability (see Section \ref{sec:degenerate}, after the proof of Proposition \ref{prop:concentration}).

In the degenerate phase, the values of the first $r$ elements of $S_n^{(q_n)}$ are stochastically bounded.

\begin{theorem}[Degenerate phase]
\label{thm:maindegenerate}
Let $(q_n)_{n\geq 1}$ be a sequence satisfying $\log(1/q_n) = n/c+o(n)$ for some constant $c > 0$. Then, there exists a random variable $Z_c$ concentrated on $\{\lfloor c \rfloor - 1, \lceil c \rceil - 1 \}$ such that, for every fixed $r \geq 1$,
\begin{align*}
\left( S_n^{(q_n)}(1), \ldots, S_n^{(q_n)}(r) \right) \underset{n \to \infty}{\overset{(d)}{\to}} \left( G_1, G_1+G_2, \ldots, G_1+ \ldots + G_r \right),
\end{align*}
where the $G_i$'s are i.i.d.~random variables, each with distribution $\Geom(1/(Z_c+1))$.
\end{theorem}

On the other hand, in the nondegenerate phase, the limit of the first elements of the descent-biased permutation, after rescaling, can be expressed in terms of i.i.d.~exponential random variables:

\begin{theorem}[Nondegenerate phase]
\label{thm:mainnondegenerate}
Let $(q_n)_{n \geq 1}$ be a sequence of real numbers in $[0,1]$ satisfying $q_n \to 0$ and $\log q_n = o(n)$. Then we have, for every fixed $r \geq 1$,
\begin{align*}
\frac{\log(1/q_n)}{n}\left( S_n^{(q)}(1), \ldots, S_n^{(q)}(r)\right) \underset{n \to \infty}{\overset{(d)}{\to}} \left( E_1, E_1+E_2, \ldots, E_1+\ldots+E_r \right),
\end{align*}
where the $E_i$'s are i.i.d.~exponential random variables of parameter $1$.
\end{theorem}

In order to study the constant phase, we define the following Markov process $(X_i)_{i \geq 1}$ with $X_i \in [0,1]$ for every positive integer $i$.
The initial value $X_1$ has density $\frac{\log(1/q)}{1-q} q^x$, and, for all $j \geq 1$, $\left( X_{j+1} | X_j \right)$ has density proportional to
\begin{align*}
\frac{\log(1/q)}{1-q} q^x \left( q \mathds{1}_{x < X_j} + \mathds{1}_{x \geq X_j}  \right) = \frac{\log(1/q)}{1-q} q^x q^{\mathds{1}_{x<X_j}}.
\end{align*}

In particular, observe that, for any $j \geq 1$, $(X_1, \ldots, X_j)$ is absolutely continuous with respect to the Lebesgue measure on $[0,1]^j$. For $q=1$, the quotient $\frac{\log(1/q)}{1-q}$ is interpreted as $1$. In this case, the $X_i$ are all independent and follow a uniform distribution on $[0,1]$.

\begin{theorem}[Constant phase]
\label{thm:mainconstant}
Fix $r \geq 1$, and assume that $q \in (0,1]$ is constant. The following convergence holds in distribution:
\begin{align*}
\frac{1}{n}\left( S_n^{(q)}(1), \ldots, S_n^{(q)}(r)\right) \underset{n \to \infty}{\overset{(d)}{\to}} \left( X_1, \ldots, X_r \right). 
\end{align*}
\end{theorem}

\paragraph*{Random descent-biased trees.}

One of our motivations to investigate these descent-biased permutations is the study of a similar model of descent-biased trees. Fix $n \geq 1$, and let $\kT_n$ denote the set of trees with $n$ vertices labelled from $1$ to $n$, and a marked vertex called the root. For an edge $e$ of a tree $T \in \kT_n$, we write $e-, e+$ for its two endpoints, with $e-$ being the endpoint closer to the root. Moreover, let us write $\ell(x)$ for the label of a vertex $x$. We say that an edge $e$ is a descent if $\ell(e-) > \ell(e+)$.  For any tree $T \in \kT_n$, we let $d(T)$ denote the number of descents of $T$. For any $q \in [0,\infty)$, we consider the random variable $\cT_n^{(q)}$ taking its values in $\kT_n$ such that, for any tree $T \in \kT_n$,

\begin{align*}
\P \left( \cT_n^{(q)} = T \right) = \frac{1}{Y_{n,q}} q^{d(T)},
\end{align*}
where $Y_{n,q} = \sum_{T \in \kT_n} q^{d(T)}$.

The (weighted) enumeration of these trees goes back to E{\u{g}}ecio{\u{g}}lu and Remmel \cite{ER86}. For the same symmetry reasons as for permutations, we may restrict ourselves to the case $q \in [0,1]$. When $q=1$, we recover the model of Cayley trees (uniformly random rooted labelled trees, see for example \cite[Section 1.2.3]{Drmota09}), while the case $q=0$ corresponds to random recursive trees \cite[Section 1.3.1]{Drmota09}, as the labels are almost surely increasing from the root to the leaves in this case. Our goal is to characterize the local limit of the tree $\cT_n^{(q)}$, that is, the random tree $\cT_*^{(q)}$ whose $k$-neighbourhood of the root is the limit of the $k$-neighbourhood of the root of $\cT_n^{(q)}$ for all fixed $k$ (see Section \ref{sec:trees} for more details). It turns out that, if $q_n \to 0$, the tree $\cT_n^{(q_n)}$ does not admit a local limit in the usual sense. However, when $q$ is constant, it is possible to use Theorem \ref{thm:mainconstant} to characterize the local limit of $\cT_n^{(q)}$.

\begin{theorem}
\label{thm:locallimittree}
For fixed $q \in (0,1]$, the sequence of random trees $\left(T_n^{(q)}\right)_{n \geq 1}$ admits a local limit, which is one-ended.
\end{theorem}

The existence and characterization of local limits of similar models of weighted random trees have recently been investigated by Durhuus and \"Unel \cite{DU22,DU23,DUpre}, but with the weight based on the height rather than the number of descents. A somewhat similar tree model (though different aspects---specifically the height---are studied) also occurs in a recent paper by Addario-Berry and Corsini \cite{ABC21}, where a binary search tree is built from a $q$-Mallows permutation.

\paragraph*{Overview of the paper.}

Section \ref{sec:degenerate} is devoted to the degenerate phase of descent-biased permutations and the proof of Theorem \ref{thm:maindegenerate}. We introduce and investigate generating functions associated with descent-biased permutations in Section \ref{sec:genfunc}, before proving Theorems \ref{thm:mainnondegenerate} and \ref{thm:mainconstant} in Section~\ref{sec:proofs}. In Section \ref{sec:mesoscopic}, we refine Theorem \ref{thm:maindegenerate} and show that, in the nondegenerate regime, the first elements of descent-biased permutations admit a nontrivial deterministic scaling limit. In Section~\ref{sec:genfun_trees}, we use analysis of generating functions associated with descent-biased trees to prove auxiliary results, which are finally used in Section \ref{sec:trees} to prove Theorem \ref{thm:locallimittree}.

\section{The degenerate phase}
\label{sec:degenerate}

Our goal in this section is to prove Theorem \ref{thm:maindegenerate}. To this end, the main idea is to consider permutations with a given number of descents. In this section, we fix $c>0$ and consider a sequence $(q_n)$ such that $\log(1/q_n) = n/c+o(n)$.

\subsection{Concentration}

Let us first show that the number of descents of $S_n^{(q_n)}$ is concentrated on at most two values in the regime we are considering in this section.

\begin{proposition}
\label{prop:concentration}
Fix $c>0$, and assume that $\log(1/q_n) = n/c+o(n)$ as $n \to \infty$. Then, with high probability, the number of descents of $S_n^{(q_n)}$ is either $\lfloor c \rfloor - 1$ or $\lceil c \rceil - 1$.
\end{proposition}

Let $\kD_n^k$ denote the set of permutations of $\kS_n$ with exactly $k$ descents.
To show Proposition \ref{prop:concentration}, we consider the random permutation $D_n^k$ defined as follows. Let $A_1, \ldots, A_n$ be $n$ i.i.d.~random variables, uniform on $\{ 1, \ldots, k+1 \}$. For all $1 \leq i \leq k+1$, set $F_i := \{ j \in \llbracket 1, n \rrbracket, A_j=i \}$, and let $L_i$ be the list of the elements of $F_i$ sorted in increasing order. Finally, let $D_n^k$ be the concatenation of $L_1, \ldots, L_{k+1}$. It turns out that the permutation $D_n^k$ is almost uniform on $\kD_n^k$:

\begin{lemma}
\label{lem:permufix}
For any fixed integer $k \geq 0$, let $E_n^k$ be the event that $D_n^k$ has $k$ descents. Then, the following hold.
\begin{itemize}
\item[(i)] The event $E_n^k$ occurs asymptotically almost surely:
\begin{align*}
\P \left( E_n^k \right) \underset{n \to \infty}{\to} 1.
\end{align*}

\item[(ii)] 
The permutation $D_n^k$ conditionally on $E_n^k$ is uniform in $\kD_n^k$.
\end{itemize}
\end{lemma}

\begin{proof}[Proof of Lemma \ref{lem:permufix}]
Let us first prove (i). For all $1 \leq i \leq k+1$, and any $x \in (1,n)$, we have
\begin{align*}
\P \left( \min \{ j, j \in F_i \} > x  \right) = \left(\frac{k}{k+1}\right)^{\lfloor x \rfloor},\quad \P \left( \max \{ j, j \in F_i \} < x  \right) = \left(\frac{k}{k+1}\right)^{\lfloor n+1-x \rfloor}.
\end{align*}
In particular, we have with high probability, jointly for all $ 1\leq i \leq k+1$:
\begin{align*}
\min F_i \leq n/4,\quad \max F_i \geq 3n/4. 
\end{align*}
Furthermore, if this event occurs, it is clear by construction that $D_n^k$ has exactly $k$ descents. This shows (i). 
In order to prove (ii), observe that, for any permutation $\tau_n$ with $k$ descents, we have
\begin{align*}
\P\left(D_n^k=\tau_n \right) = (k+1)^{-n},
\end{align*}
as this event is uniquely determined by the values of the random variables $A_1, \ldots, A_n$. Hence, conditionally on $E_n^k$, the permutation $D_n^k$ is uniform in $\kD_n^k$.
\end{proof}

As a consequence, we get the following (well-known) corollary. Let $N(k,n)$ be the number of permutations in $\kS_n$ with $k$ descents, which is also known as an Eulerian number. 

\begin{cor}
\label{cor:fixeddescents}
Let $k \geq 1$ be a fixed integer. Then
\begin{itemize}
\item[(i)]$N(k,n) \leq (k+1)^n$.
\item[(ii)] $N(k,n) = (k+1)^n(1+o(1))$ as $n \to \infty$.
\end{itemize}
\end{cor}

\begin{proof}[Proof of Corollary \ref{cor:fixeddescents}]
In order to prove (i), consider the map
\begin{align*}
f_k: \{1, \ldots, k+1 \}^n \to \kS_n,
\end{align*}
which associates to an $n$-tuple $(a_1, \ldots, a_n)$ the permutation $(D_n^k \,:\, A_i=a_i, 1 \leq i \leq n)$. The map $f_k$ is clearly injective, and $\kD_n^k \subseteq \operatorname{Im}(f_k)$, so that $N(k,n)=|\kD_n^k| \leq (k+1)^n$. The approximation (ii) comes directly from Lemma \ref{lem:permufix}.
\end{proof}

We can now prove Proposition \ref{prop:concentration}.

\begin{proof}[Proof of Proposition \ref{prop:concentration}]
For any $k \geq 0$, let $w_{k,n} := \sum_{\tau \in \kD_n^k} q_n^{d(\tau)}$ be the total weight of all permutations of $\llbracket 1, n \rrbracket$ with $k$ descents. Then, $w_{n,k}=(k+1)^n q_n^k (1+o(1))$ by Corollary \ref{cor:fixeddescents} (ii), which can be rewritten as
\begin{align*}
w_{n,k} = \exp \left[ n \left( \log (k+1) - k/c +o(1) \right) \right].
\end{align*}
On the other hand, for all $k$, $w_{n,k} \leq (k+1)^n q_n^k$ by Corollary \ref{cor:fixeddescents} (i). Define $K_c$ to be the set of all nonnegative integers $k$ for which $\log (k+1) - k/c$ attains its maximum. We immediately get that 
\begin{align*}
\sum_{0 \leq k \leq n-1} w_{n,k} = \sum_{k \in K_c} w_{n,k} (1+o(1)).
\end{align*}
Proposition \ref{prop:concentration} follows, since $K_c \subseteq \{ \lfloor c \rfloor - 1, \lceil c \rceil - 1\}$. 
\end{proof}

Observe finally that, since $w_{n,k} \leq (k+1)^n q_n^k$, with high probability $S_n^{(q_n)}$ has no descent for $\log(1/q_n) \gg n$ (even under the weaker condition that $2^n q_n \to 0$). Thus with high probability $S_n^{(q_n)}$ is the identity permutation in this case. Hence, the degenerate case that we consider, where $\log(1/q_n) = n/c+o(n)$ for some constant $c > 0$, is the smallest threshold to see non-trivial permutations appear asymptotically.


\subsection{Local limit}

Thanks to Proposition \ref{prop:concentration}, to study the local limit of $S_n^{(q_n)}$ when $\log(1/q_n) = n/c+o(n)$, we only need to understand the local limit of a permutation conditioned on having a fixed number $k$ of descents.

\begin{proof}[Proof of Theorem \ref{thm:maindegenerate}]
Proving Theorem \ref{thm:maindegenerate} boils down to investigating the behaviour of the first $r$ elements of the permutation $D_n^k$ for fixed $k$. To this end, observe that they are simply the first $r$ elements $i_1 < \ldots < i_r$ of $\llbracket 1, n \rrbracket$ satisfying $A_{i_1} = \ldots = A_{i_r} = 1$ (we know that with high probability, uniformly for $k$ in a compact subset of $\R$, $|F_1| \geq \sqrt{n} \geq r$, so that the first $r$ elements of $S_n^{(q_n)}$ belong to $F_1$). Theorem \ref{thm:maindegenerate} then follows immediately from the definition of the $A_i$'s.
\end{proof}

\section{Generating functions}
\label{sec:genfunc}

In order to study the nondegenerate and constant phases, we use the machinery of generating functions. We first obtain a differential equation satisfied by the generating function of descent-biased permutations, then use it to obtain asymptotics of the moments of the first element. In the next section, we use it to deduce Theorems \ref{thm:mainnondegenerate} and \ref{thm:mainconstant} from these asymptotics.

\subsection{Generating function of descent-biased permutations}\label{sec:gf_descent-biased}

For any $k \in \llbracket 1, n \rrbracket$, let $P_n^k(q)$ be the number of permutations of $\llbracket 1, n \rrbracket$, weighted by $q^{d(\sigma)}$, whose first element is $k$. We obtain the following relation by distinguishing the cases that the second element is less than $k$ and that the second element is greater than $k$, respectively:
\begin{align*}
P_n^k(q) &= q \sum_{j=1}^{k-1} P_{n-1}^j(q) + \sum_{j=k}^{n-1} P_{n-1}^j(q)\\
&= q \sum_{j=1}^{n-1} P_{n-1}^j(q) + (1-q) \sum_{j=k}^{n-1} P_{n-1}^j(q).
\end{align*}
Writing $R_n(q,y) = \sum_{k = 1}^n P_n^k(q) y^k$, we get
\begin{align*}
R_n(q,y) &= q \sum_{k = 1}^{n} y^k \sum_{j =1}^{n-1} P_{n-1}^j(q) +(1-q)  \sum_{k = 1}^{n} y^k \sum_{j=k}^{n-1} P_{n-1}^j(q)\\
&= qy \frac{1-y^n}{1-y} R_{n-1}(q,1) + (1-q) \sum_{j=1}^{n-1} P_{n-1}^j(q) \sum_{k=1}^j y^k\\
&= qy \frac{1-y^n}{1-y} R_{n-1}(q,1) + (1-q)y \sum_{j=1}^{n-1} P_{n-1}^j(q) \frac{1-y^j}{1-y}\\
&= qy \frac{1-y^n}{1-y} R_{n-1}(q,1) + \frac{(1-q)y}{1-y} \sum_{j=1}^{n-1} P_{n-1}^j(q) - \frac{(1-q)y}{1-y} \sum_{j=1}^{n-1} P_{n-1}^j(q) y^j \\
&= \frac{y}{1-y} R_{n-1}(q,1) - \frac{qy^{n+1}}{1-y} R_{n-1}(q,1) -  \frac{(1-q)y}{1-y} R_{n-1}(q,y).
\end{align*}
This recursion still holds for $n=1$ if we set $R_0(q,y) = y$ (i.e., interpret the empty permutation as starting with $1$). Now let $S(x,y) := \sum_{n \geq 0} R_n(q,y) \frac{x^n}{n!}$ be the bivariate (exponential) generating function. The recursion for $R_n$ translates to the differential equation
\begin{align*}
\frac{\partial}{\partial x} S(x,y) = \frac{y}{1-y} \big( S(x,1) - qy S(xy,1) - (1-q) S(x,y) \big).
\end{align*}

The function $S(x,1)$ is the well-known generating function of the Eulerian polynomials (see  \cite[p.~34--45]{Knu73}):
\begin{equation}
\label{eq:euler}
S(x,1) = \frac{1-q}{e^{(q-1)x}-q}.
\end{equation}

We finally get
\begin{equation}
\label{eq:equadiff}
\frac{\partial}{\partial x} S(x,y) = \frac{(1-q) y}{1-y} \left( \frac{1}{e^{(q-1)x}-q} - \frac{qy}{e^{(q-1)xy}-q} - S(x,y) \right).
\end{equation}

It is possible to derive an integral representation for $S(x,y)$ from this differential equation, but for our purposes it is actually more convenient to work directly with the differential equation in the following.

\subsection{Moments}

We will be particularly interested in $M_r(x) = (\frac{\partial}{\partial y})^r S(x,y)|_{y=1}$, which will yield us the $r$-th factorial moment of the distribution of the first element $S_n^{(q)}(1)$ of a permutation. From~\eqref{eq:equadiff}, we obtain a recursion for these functions by first multiplying by $1-y$, then differentiating $r$ times with respect to $y$ and setting $y=1$: letting
$$f_r(x) = \Big( \frac{\partial}{\partial y} \Big)^r \Big( \frac{1}{e^{(q-1)x}-q} - \frac{qy}{e^{(q-1)xy}-q} \Big) \Big|_{y=1},$$
we have
$$r M_{r-1}'(x) = (1-q) \big(M_r(x) + r M_{r-1}(x) - f_r(x) - r f_{r-1}(x)\big).$$
This can be rewritten as follows:
\begin{equation}\label{eq:moment_recursion}
M_r(x) = \frac{r}{1-q} M_{r-1}'(x) - r M_{r-1}(x) + f_r(x) + rf_{r-1}(x).
\end{equation}
We will inductively derive the asymptotic behaviour of the coefficients of $M_r(x)$ from this recursion. For fixed $q$, this would be a fairly straightforward exercise in singularity analysis (which we will use later in Section~\ref{sec:genfun_trees} in the analysis of descent-biased trees), but to obtain results that remain valid as $q \to 0$, we need to apply a somewhat different approach.

\begin{theorem}
\label{thm:momentsbehaviour}
Suppose that $0 < q = q_n < 1$, where $q_n$ is either constant or $q_n \to 0$ with $\log(1/q_n) =o(n)$.
For every fixed positive integer $r$, we have
$$[x^n] M_r(x) \sim \mu_r(q_n) n^r [x^n] f_0(x),$$
where
$$\mu_r(q) =  \frac{\log(1/q)}{1-q}  \int_0^1 x^rq^x\,dx$$
is the $r$-th moment of a random variable with density $\frac{\log(1/q)}{1-q} q^x \mathds{1}_{[0,1]}(x)$.
\end{theorem}

\begin{remark}
It is easy to show, using integration by parts, that for any $q \in (0,1)$, the moment sequence $(\mu_r(q))_{r \geq 0}$ satisfies the following recursion:

\begin{equation}\label{eq:momentrec}
\mu_0(q)=1, \qquad\mu_r(q) = \frac{r \mu_{r-1}(q)}{\log(1/q)} - \frac{q}{1-q}.
\end{equation}
\end{remark}

As a corollary, we obtain the limiting distribution of $S_n^{(q)}(1)$ for fixed $q$.

\begin{cor}
\label{cor:expectation}
Suppose that $0 < q = q_n < 1$, where $q_n$ is either constant or $q_n \to 0$ with $\log(1/q_n) =o(n)$.
For every fixed positive integer $r$, we have, as $n \to \infty$:
$$\E \left( \left( \frac{S_n^{(q_n)}(1)}{n} \right)^r \right) \sim \mu_r(q_n).$$
In particular, for fixed $q_n = q$, $\frac{S_n^{(q)}(1)}{n}$ converges in distribution to a random variable whose moments are $\mu_r(q)$, i.e., a random variable with density $\frac{\log(1/q)}{1-q} q^x \mathds{1}_{[0,1]}(x)$.
\end{cor}

\begin{proof}
Note that, since $f_0(x)=S(x,1)$, $\frac{[x^n] M_r(x)}{[x^n] f_0(x)}$ is the $r$-th factorial moment of $S_n^{(q_n)}(1)$, which is asymptotically equivalent to $\mu_r(q_n) n^r$ by Theorem \ref{thm:momentsbehaviour}. It is easy to show from the recursion~\eqref{eq:momentrec} that 
$$\frac{\mu_{r-1}(q_n)n^{r-1}}{\mu_r(q_n) n^r}  =O(n^{-1}\max \{\log(1/q_n),1\}) = o(1).$$
Thus the $r$-th factorial moment is also asymptotically equivalent to the $r$-th moment of $S_n^{(q_n)}(1)$, and the first statement follows. Since the distribution whose density is $\frac{\log(1/q)}{1-q} q^x \mathds{1}_{[0,1]}(x)$ is uniquely determined by its moments as it has bounded support, convergence in distribution follows immediately as well (see e.g.~\cite[Theorem C.2]{FS09}).
\end{proof}

\subsection{Proof of Theorem \ref{thm:momentsbehaviour}}

We first consider the coefficients of $[x^n] f_r(x)$.

\begin{lemma}\label{lem:fr}
For $r \geq 1$, we have (with $a^{\underline{r}} = a(a-1)\cdots(a-r+1)$ denoting the falling factorial)
$$[x^n] f_r(x) = (n+1)^{\underline{r}} [x^n] \frac{q}{q-e^{(q-1)x}} = -\frac{q}{1-q} (n+1)^{\underline{r}} [x^n] f_0(x).$$
\end{lemma}

\begin{proof}
Let $\sum_{n \geq 0} a_nx^n$ be the series expansion of $f_0(x) = \frac{1-q}{e^{(q-1)x}-q}$. Replacing $x$ by $xy$, we obtain
$$\frac{q}{e^{(q-1)xy}-q} = \frac{q}{1-q} \sum_{n \geq 0} a_nx^n y^n$$
and thus
$$ - \frac{qy}{e^{(q-1)xy}-q} = - \frac{q}{1-q} \sum_{n \geq 0} a_nx^ny^{n+1}.$$
The result follows by differentiating $r$ times with respect to $y$ and plugging in $y=1$.
\end{proof}

Next, we need a technical lemma on the coefficients of the generating function $f_0(x) = \frac{1-q}{e^{(q-1)x}-q}$.

\begin{lemma}\label{lem:f0}
Suppose that $0 < q = q_n < 1$, where $q_n$ is either constant or $q_n \to 0$ with $\log(1/q_n) =o(n)$. Then we have
$$\frac{[x^{n+1}] f_0(x)}{[x^n] f_0(x)} \sim \frac{1-q_n}{\log(1/q_n)}$$
as $n \to \infty$. The statement also holds for $q_n=1$ if the right side is interpreted as $1$.
\end{lemma}

\begin{proof} Throughout the proof, we only write $q$ instead of $q_n$ for simplicity. 
For $q = 1$, we have
$$f_0(x) = \lim_{q \to 1} \frac{1-q}{e^{(q-1)x}-q} = \frac{1}{1-x},$$
and the statement is trivial. So assume that $q < 1$. 
Now we have
\begin{align*}
[x^n] f_0(x) &= [x^n] \frac{(1-q)e^{(1-q)x}}{1-q e^{(1-q)x}} = (1-q) [x^n] \sum_{k \geq 1} q^{k-1} e^{k(1-q)x} \\
&= (1-q) \sum_{k \geq 1} q^{k-1} \frac{k^n (1-q)^n}{n!} = \frac{(1-q)^{n+1}}{q n!} \sum_{k \geq 1} q^k k^n.
\end{align*}
Let us use the abbreviation $Q = \frac{n}{\log(1/q)}$. By our assumptions on $q$, we know that $Q \to \infty$ as $n \to \infty$.
We can apply the discrete Laplace method (see \cite[Appendix B.6]{FS09}) to the sum: set $k = Q(1+t) = \frac{n}{\log(1/q)} (1+t)$, and observe that
$$q^k k^n = e^{-n(1+t)} Q^n (1+t)^n = \frac{Q^n}{e^n} \big( (1+t) e^{-t} \big)^n.$$
The function $t \mapsto (1+t)e^{-t}$ is increasing for $-1 \leq t \leq 0$ and decreasing for $t \geq 0$, and $(1+t)e^{-t} = e^{-t^2/2 + O(t^3)}$ around $0$.
We can use standard estimates to show that only those terms where $t$ is close to $0$ are asymptotically relevant. Specifically, set $T = Q^{-1/3}$. We show that all terms with $|t| \geq T$ are negligible. First, we note that for the value of $k$ that is closest to $Q$ (call it $k_0:=Q(1+t_0)$), we must have $|k_0 - Q| \leq \frac12$, thus $|t_0| \leq \frac{1}{2Q}$. For this specific choice, we have
$$q^{k_0} k_0^n \geq \frac{Q^n}{e^n} e^{-nQ^{-2}/8 + O(nQ^{-3})}.$$
The sum over all integers $k$ for which the corresponding $t$ satisfies $T \leq |t| \leq 1$ (there are $O(Q)$ such integers) can be bounded as follows:
$$\sum_{\substack{k = Q(1+t) \\ T \leq |t| \leq 1}} q^k k^n = O \Big( Q \cdot \frac{Q^n}{e^n} e^{-nT^2/2 + O(nT^3)} \Big).$$
By our choice of $T$, we have $n T^2  = n Q^{-2/3} \gg nQ^{-2}$. Moreover, $nT^2 = n Q^{-2/3} \gg \log Q$, since $Q = O(n)$. Thus this entire sum is negligible compared to the single term $q^{k_0}k_0^n$. Finally, for the terms where $|t| > 1$, we use the simple inequality $(1+t)e^{-t} \leq e^{-t/4}$ to bound the sum as follows:
$$\sum_{\substack{k = Q(1+t) \\ |t| > 1}} q^k k^n \leq \sum_{\substack{k = Q(1+t) \\ |t| > 1}} \frac{Q^n}{e^n} e^{-nt/4} = O \Big( Q \cdot \frac{Q^n}{e^n} e^{-n/4} \Big)$$
by summing the geometric series. Since $Q \to \infty$, this is again negligible compared to the single term $q^{k_0}k_0^n$. So it follows that only terms in the sum with $|t| \leq T$ (equivalently, $|k - Q| \leq QT$) are relevant. This gives us
$$[x^n] f_0(x) \sim \frac{(1-q)^{n+1}}{q n!} \sum_{|k - Q| \leq QT} q^k k^n,$$
and by the same reasoning
$$[x^{n+1}] f_0(x) \sim \frac{(1-q)^{n+2}}{q (n+1)!} \sum_{|k - Q| \leq QT} q^k k^{n+1}.$$
Now the quotient of corresponding terms in the two sums is always $k = Q (1+O(T)) = \frac{n}{\log(1/q)} (1 + o(1))$, and it follows that
$$\frac{[x^{n+1}] f_0(x)}{[x^n] f_0(x)} \sim \frac{1-q}{n+1} \cdot \frac{n}{\log(1/q)} \sim \frac{1-q}{\log(1/q)},$$
completing the proof.
\end{proof}

Now we can finally obtain a more complete picture of the asymptotic behaviour of $M_r(x)$.

\begin{proof}[Proof of Theorem \ref{thm:momentsbehaviour}]
Let us again drop the dependence of $q$ on $n$ and only write $q$ instead of $q_n$.
We prove the statement by induction on $r$. For $r=1$,~\eqref{eq:moment_recursion} gives us
$$[x^n] M_1(x) = \frac{1}{1-q} [x^n] M_0'(x) + [x^n] f_1(x),$$
since $M_0(x) = f_0(x) = S(x,1)$. Now Lemma~\ref{lem:fr} implies that
$$[x^n] M_1(x) = \frac{n+1}{1-q} [x^{n+1}] M_0(x) - \frac{q(n+1)}{1-q} [x^n] f_0(x) = \frac{n+1}{1-q} \big( [x^{n+1}] f_0(x) - q [x^n] f_0(x) \big).$$
Lastly, we apply Lemma~\ref{lem:f0} to obtain
$$[x^n] M_1(x) \sim \frac{n}{1-q} \Big( \frac{1-q}{\log(1/q)} - q \Big) [x^n] f_0(x),$$
which does indeed agree with the statement, as $\mu_1(q) = \frac{1}{\log(1/q)} - \frac{q}{1-q}$. For the induction step, we note that, by~\eqref{eq:moment_recursion},
$$[x^n] M_r(x) = \frac{r}{1-q} [x^n] M_{r-1}'(x) - r [x^n] M_{r-1}(x) + [x^n] f_r(x) + r [x^n] f_{r-1}(x).$$
By the induction hypothesis and Lemma~\ref{lem:f0}, we have
$$[x^n] M_{r-1}'(x) = (n+1) [x^{n+1}] M_{r-1}(x) \sim n^r \mu_{r-1}(q) [x^{n+1}] f_0(x) \sim n^r \mu_{r-1}(q) \frac{1-q}{\log(1/q)} [x^n] f_0(x)$$
as well as
$$[x^n] M_{r-1}(x)  \sim n^{r-1}  \mu_{r-1}(q) [x^n] f_0(x).$$
Moreover, by Lemma~\ref{lem:fr},
$$[x^n] f_r(x) \sim - \frac{q}{1-q} n^r [x^n] f_0(x) \qquad \text{and} \qquad [x^n] f_{r-1}(x) \sim - \frac{q}{1-q} n^{r-1} [x^n] f_0(x).$$
Consequently, only the terms involving $M_{r-1}'(x)$ and $f_r(x)$ are asymptotically relevant, and we obtain
$$[x^n] M_r(x) \sim \frac{r}{1-q} \cdot n^r \mu_{r-1}(q) \frac{1-q}{\log(1/q)} [x^n] f_0(x) - \frac{q}{1-q} n^r [x^n] f_0(x).$$
In view of~\eqref{eq:momentrec}, this completes the induction.
\end{proof}

\section{Proof of Theorems \ref{thm:mainnondegenerate} and \ref{thm:mainconstant}}\label{sec:proofs}

We now have all the ingredients to prove Theorems \ref{thm:mainnondegenerate} and \ref{thm:mainconstant}. Let us start with the former. We first need the following lemma:

\begin{lemma}
\label{lem:integrals}
For any fixed $r \geq 0$, we have, as $x \to 0^+$:

\begin{align*}
\int_0^\infty x^v v^r dv \sim \int_0^1 x^v v^r dv \sim \frac{r!}{(-\log x)^{r+1}}.
\end{align*}
\end{lemma}

\begin{proof}
Setting $I_r:= \int_{v=0}^\infty x^v v^r dv$ and $J_r :=  \int_0^1 x^v v^r dv$, we have $I_0=\frac{1}{-\log x}$, $J_0=\frac{1-x}{-\log x}$, and for all $r \geq 1$:
\begin{align*}
I_r = \frac{r}{-\log x} I_{r-1}, J_r = \frac{x}{\log x} + \frac{r}{-\log x} J_{r-1}.
\end{align*}

Hence, for any fixed $r$, as $x \to 0^+$, we have $I_r \sim J_r \sim \frac{r!}{(-\log x)^{r+1}}$.
\end{proof}

\begin{proof}[Proof of Theorem \ref{thm:mainnondegenerate}]
By Corollary \ref{cor:expectation} and Lemma \ref{lem:integrals}, we get
\begin{equation}
\label{eq:forr=1}
\frac{\log(1/q_n)}{n} S_n^{(q_n)}(1) \underset{n \to \infty}{\overset{(d)}{\to}} X_1,
\end{equation}
where $X_1 \sim \Exp(1)$, since the $r$-th moment converges to $r!$, which is the $r$-th moment of an $\Exp(1)$-distribution, for every $r$.

In order to generalize to $(S_n^{(q_n)}(1), \ldots, S_n^{(q_n)}(r))$, observe first that, with high probability, 
\begin{equation}
\label{eq:ordered}
S_n^{(q_n)}(1) < S_n^{(q_n)}(2) < \ldots < S_n^{(q_n)}(r).
\end{equation}

\noindent Indeed, fix $R \geq r+1$. Conditionally on $(S_n^{(q_n)}(i))_{1 \leq i \leq R}$, $(S_n^{(q_n)}(i))_{R+1 \leq i \leq n}$ only depends on $S_n^{(q_n)}(R)$. Hence, conditionally on the value of $S_n^{(q_n)}(R)$ and on the set of values $(S_n^{(q_n)}(i))_{1 \leq i \leq R-1}$, it is clear that with high probability, the first $R$ elements of $S_n^{(q_n)}$ induce at most one descent (since $q_n \to 0$). If there is no descent, then we get \eqref{eq:ordered}. Otherwise, conditionally on exactly one descent occurring, using the same argument as in the proof of Lemma \ref{lem:permufix}, the $\limsup$ of the probability that it occurs at one of the first $r$ positions goes to $0$ as $R \to \infty$. This also implies \eqref{eq:ordered}.

Now observe that, conditionally on $S_n^{(q_n)}(1)$ and on the event $S_n^{(q_n)}(1) < S_n^{(q_n)}(2)$, $S_n^{(q_n)}(2)$ is distributed as $1+T_{n-1,S_n^{(q_n)}(1)}^{(q_{n})}$, where $T_{n,k}^{(q_n)}$ is an independent copy of $S_n^{(q_n)}$ conditioned on $T_{n,k}^{(q_n)}(1) \geq k$. For any $r > 0$, and any sequence $(k_n, n \geq 1)$ such that $\frac{\log(1/q_n)}{n} k_n \underset{n \to \infty}{\to} r$, we have by \eqref{eq:forr=1}:
\begin{equation}
\label{eq:condit}
\frac{\log(1/q_n)}{n} T_{n,k_n}^{(q_n)}(1) \underset{n \to \infty}{\overset{(d)}{\to}} \left( E' \big| E' \geq r \right).
\end{equation}
Furthermore, for all $r \geq 0$, we have that $\left( E' \big| E' \geq r\right) \overset{(d)}{=} r+E''$, where $E'' \sim \Exp(1)$.

For convenience, write $(S_1, S_2)$ for $(S_n^{(q_n)}(1), S_n^{(q_n)}(2))$. Writing $\R_+$ for the set of nonnegative real numbers, let $G: \R_+^2 \to \R$ be continuous and bounded. For any $x \geq 0$, we have by \eqref{eq:condit} that
\begin{align*}
g_n: x \mapsto \E\left[G(S_1, S_2-S_1) \big| S_1<S_2, S_1 = \left\lfloor \frac{n}{\log(1/q_n)} x \right\rfloor\right]
\end{align*}
converges pointwise to $\E[G(x,E)]$, where $E \sim \Exp(1)$ is independent of the value $x$. Hence, using \eqref{eq:forr=1}, we get by dominated convergence that, for any $a>0$:
\begin{align*}
\int_0^a g_n(x) dx \to \E[G(E,E') \mathds{1}_{E < a}],
\end{align*}
where $E, E'$ are two independent $\Exp(1)$ variables. Letting $a \to \infty$, we get
\begin{align*}
(S_1, S_2) \underset{n \to \infty}{\overset{(d)}{\to}} (E,E+E').
\end{align*}

This completes the proof of Theorem \ref{thm:mainnondegenerate} in the case $r=2$. The general case is proved in the same way.
\end{proof}

We conclude by dealing with the case where $q \in [0,1]$ is constant. By Corollary \ref{cor:expectation}, letting $X$ be a random variable with density $\frac{q^x dx}{\int_0^1 q^v dv}$ on $[0,1]$, we get that

\begin{align*}
\frac{S_n^{(q)}(1)}{n} \underset{n \to \infty}{\overset{(d)}{\to}} X.
\end{align*}

Let  $D^{(r)}$ denote the density of $(X_1, \ldots, X_r)$, where $(X_i, i\geq 1)$ is the process introduced in Theorem \ref{thm:mainconstant}. For $n \geq 1$, we write $A_n(q)$ for the $n$-th Eulerian polynomial:
\begin{align*}
A_n(q) = \sum_{k=1}^n P_n^k(q),
\end{align*}
i.e., $A_n(q)$ is the weighted number of permutations of size $n$, where the weight associated with a permutation $\sigma$ is $q^{d(\sigma)}$. In the notation of Section~\ref{sec:gf_descent-biased}, we have $A_n(q) = R_n(q,1)$. Finally, set
\begin{align*}
C_q = \frac{\log(1/q)}{1-q}.
\end{align*}

\begin{lemma}
\label{lem:densityconstant}
For any $(x_1, \ldots, x_r) \in [0,1]^r$,
\begin{align*}
D^{(r)}(x_1, \ldots, x_r) = C_q^r \frac{r!}{A_r(q)} q^{\sum_{i=1}^r x_i} q^{d(\sigma_r)},
\end{align*}
where $\sigma_r$ is the permutation of $\llbracket 1, r \rrbracket$ induced by $x_1, \ldots, x_r$.
\end{lemma}

\begin{proof}
We know that $D^{(r)}(x_1, \ldots, x_r)$ is proportional to $q^{x_1} \prod_{i=2}^r (q^{x_i} q^{\mathds{1}_{x_i<x_{i-1}}}) = q^{\sum_{i=1}^r x_i} q^{d(\sigma_r)}$.
Hence, the only thing left to check is that
\begin{equation}
\label{eq:density}
I := \int_{x_1, \ldots, x_r=0}^1 q^{x_1} \prod_{i=2}^r \Big( q^{x_i} q^{\mathds{1}_{x_i<x_{i-1}}} \Big) \prod_{i=1}^r dx_i = C_q^{-r} \frac{A_r(q)}{r!}.
\end{equation}
To see this, for any $\sigma \in \kS_r$, set 
\begin{align*}
I_\sigma := \int_{x_1, \ldots, x_r=0}^1 q^{x_{\sigma(1)}} \prod_{i=2}^r \Big( q^{x_{\sigma(i)}} q^{\mathds{1}_{x_{\sigma(i)}<x_{\sigma(i-1)}}} \Big) \prod_{i=1}^r dx_i.
\end{align*}
Then, by symmetry, for all $\sigma \in \kS_r$, $I_\sigma = I$. Furthermore, we have 
\begin{align*}
r! I = \sum_{\sigma \in \kS_r} I_\sigma &= \sum_{\sigma \in \kS_r} \int_{x_1, \ldots, x_r=0}^1 q^{\sum_{i=1}^r x_i} q^{d(\sigma_r \circ \sigma)} \prod_{i=1}^r dx_i\\
&= A_r(q) \prod_{i=1}^r \int_0^1 q^x dx\\
&= A_r(q) C_q^{-r},
\end{align*}
using the fact that $\int_0^1 q^x dx = C_q^{-1}$. This implies \eqref{eq:density} and thus Lemma \ref{lem:densityconstant}.
\end{proof}

Let us now state a result that will also be useful in the next section.

\begin{proposition}
\label{prop:loclim1}
Fix $q \in (0,1]$. Then, the following hold:
\begin{itemize}
\item[(i)] Uniformly for all $1 \leq k \leq n$,  we have
\begin{align*}
n \P \left(S_n^{(q)}(1)=k \right) \underset{n \to \infty}{\sim} \frac{\log (1/q)}{1-q} q^{k/n}.
\end{align*}

\item[(ii)]
For any $r \geq 1$, uniformly for any distinct positive integers $k_1,\ldots,k_r \in \llbracket 1, n \rrbracket$, we have
\begin{align*}
n^r \P\left( S_n^{(q)}(1)=k_1, \ldots, S_n^{(q)}(r)=k_r \right) \underset{n \to \infty}{\sim} D^{(r)}\left( \frac{k_1}{n}, \ldots, \frac{k_r}{n} \right),
\end{align*}
where we recall that $D^{(r)}: [0,1]^r \to [0,\infty)$ is the density of $(X_1, \ldots, X_r)$ in Theorem \ref{thm:mainconstant}, given by Lemma \ref{lem:densityconstant}. 
\end{itemize} 
\end{proposition}

Observe that this immediately implies Theorem \ref{thm:mainconstant} as well. For $q = 1$, the quotient $\frac{\log (1/q)}{1-q}$ is again interpreted as $1$. In this case, the statement becomes trivial since $S_n^{(1)}$ is simply a uniformly random permutation.

\begin{proof}[Proof of Proposition \ref{prop:loclim1}]
Let us first prove (i).
Recall that we have
\begin{align*}
\frac{S_n^{(q)}(1)}{n} \underset{n \to \infty}{\overset{(d)}{\to}} Y_q,
\end{align*}
where $Y_q$ is a random variable on $[0,1]$ with density $C_q q^v$, for $C_q := \left( \int_0^1 q^v dv \right)^{-1}=\frac{\log(1/q)}{1-q}$. Let us now fix $x \in (0,1)$ and prove that 
\begin{align*}
\P\left( S_n^{(q)}(1) = \lfloor xn \rfloor \right) \underset{n \to \infty}{\sim} C_q q^x.
\end{align*}
To show this, just observe that by monotonicity, for sufficiently small $\epsilon>0$, we have
\begin{align*}
\P\left( S_n^{(q)}(1) = \lfloor xn \rfloor \right) \geq \left(\lfloor n (x+\epsilon)\rfloor - \lfloor nx \rfloor\right) \P\left( S_n^{(q)}(1) \in [xn, (x+\epsilon)n] \right), 
\end{align*}
so that 
\begin{align*}
\liminf_{n \to \infty} \frac{\P\left( S_n^{(q)}(1) = \lfloor xn \rfloor \right)}{n} \geq C_q q^x.
\end{align*}
The same argument provides 
\begin{align*}
\limsup_{n \to \infty} \frac{\P\left( S_n^{(q)}(1) = \lfloor xn \rfloor \right)}{n} \leq C_q q^x.
\end{align*}

Also by the same argument, this convergence holds jointly for any $k$-tuple $(x_1, \ldots, x_k) \in [0,1]^k$. 
By monotonicity again, we obtain Proposition \ref{prop:loclim1} (i) if we can prove that this estimate also holds for $k=1$ and $k=n$. Observe that
\begin{align*}
\P \left( S_n^{(q)}(1)=1 \right) = \frac{[x^{n-1}] S(x,1)}{[x^n] S(x,1)} = q^{-1} \P \left( S_n^{(q)}(1) =n \right).
\end{align*}

Since the density of the function at $1$ is exactly $q$ times the density of the function at $0$, we get the result.

Let us now prove (ii) by induction on $r \geq 1$. The case $r=1$ corresponds to (i). Assume now that (ii) is true for some $r \geq 1$. Fix $k_1, \ldots, k_{r+1} \in \llbracket 1, n \rrbracket$, all distinct, set $p_r := \P \left( S_n^{(q)}(1)=k_1, \ldots, S_n^{(q)}(r)=k_r \right)$, and observe that
\begin{align*}
\P \left( S_n^{(q)}(1)=k_1, \ldots, S_n^{(q)}(r+1)=k_{r+1} \right) 
&= \frac{p_r}{Z_n^{(r+1)}} \left( \sum_{i=0}^{r} \mathds{1}_{\{k_r<k_{r+1}\} \cap E_i} \P \left( S_{n-r}^{(q)}(1)=k_{r+1}-i \right) \right. \\ 
&\qquad+ \left. q \sum_{i=0}^{r} \mathds{1}_{\{k_r>k_{r+1}\} \cap E_i} \P \left( S_{n-r}^{(q)}(1)=k_{r+1}-i \right)\right),
\end{align*}
where $E_i$ is the event that exactly $i$ elements among $\{k_1, \ldots, k_r \}$ are smaller than $k_{r+1}$, and $Z_n^{(r+1)}$ is a renormalizing constant. In particular, using (i), we get that
\begin{align*}
n^{r+1} \P \left( S_n^{(q)}(1)=k_1, \ldots, S_n^{(q)}(r+1)=k_{r+1} \right) &\sim \frac{p_r n^{r+1}}{Z_n^{(r+1)}}  \P \left( S_n^{(q)}(1)=k_{r+1} \right) q^{\mathds{1}_{k_{r+1}<k_r}}\\
&\sim \frac{1}{Z_n^{(r+1)}} D^{(r)}\left( \frac{k_1}{n}, \ldots, \frac{k_r}{n} \right) D^{(1)}\left( \frac{k_{r+1}}{n} \right) q^{\mathds{1}_{k_{r+1}<k_r}}.
\end{align*}
We can now use Lemma \ref{lem:densityconstant} (in particular, $D^{(1)}(x) = C_q q^x$) to get
\begin{equation}
\label{eq:Dr}
n^{r+1} \P \left( S_n^{(q)}(1)=k_1, \ldots, S_n^{(q)}(r+1)=k_{r+1} \right) \sim \frac{K_r}{Z_n^{(r+1)}} D^{(r+1)}\left( \frac{k_1}{n}, \ldots, \frac{k_{r+1}}{n} \right),
\end{equation}
where 
\begin{align*}
K_r = \frac{A_{r+1}(q)}{(r+1)A_r(q)}.
\end{align*}
To see this, just observe that, by Lemma \ref{lem:densityconstant},
\begin{align*}
D^{(r)}\left( \frac{k_1}{n}, \ldots, \frac{k_r}{n} \right) D^{(1)}\left( \frac{k_{r+1}}{n} \right) q^{\mathds{1}_{k_{r+1}<k_r}} &= C_q^{r+1} \frac{r!}{A_r(q) \, A_1(q)} q^{\sum_{i=1}^r k_i/n} q^{k_{r+1}/n} q^{d(\sigma_r)} q^{\mathds{1}_{k_{r+1}<k_r}}\\
&= C_q^{r+1} \frac{r!}{A_r(q)} q^{\sum_{i=1}^{r+1} k_i/n} q^{d(\sigma_{r+1})}
\end{align*}
and
\begin{align*}
D^{(r+1)}\left( \frac{k_1}{n}, \ldots, \frac{k_{r+1}}{n} \right) = C_q^{r+1} \frac{(r+1)!}{A_{r+1}(q)} q^{\sum_{i=1}^{r+1} k_i/n} q^{d(\sigma_{r+1})},
\end{align*}
which provides \eqref{eq:Dr}. We thus only need to prove that, for all $r$,
\begin{equation}
\label{eq:kr}
Z_n^{(r+1)} \underset{n \to \infty}{\to} K_r.
\end{equation}
To this end, just observe that
\begin{align*}
Z_n^{(r+1)} &= \sum_{k_1, \ldots, k_{r+1}} p_r \left( \sum_{i=0}^{r} \mathds{1}_{\{k_r<k_{r+1}\} \cap E_i}  \P \left( S_{n-r}^{(q)}(1)=k_{r+1}-i \right) \right. \\ 
&\quad+ \left. q \sum_{i=0}^{r} \mathds{1}_{\{k_{r+1}<k_r\} \cap E_i} \P \left( S_{n-r}^{(q)}(1)=k_{r+1}-i \right)\right)\\
&\sim \sum_{k_1, \ldots, k_{r+1}} p_r \P \left( S_n^{(q)}(1)=k_{r+1} \right) q^{\mathds{1}_{k_{r+1}<k_r}}\\
&\sim n^{-r-1} \sum_{k_1, \ldots, k_{r+1}} D^{(r)}\left( \frac{k_1}{n}, \ldots, \frac{k_r}{n} \right) D^{(1)}\left( \frac{k_{r+1}}{n} \right) q^{\mathds{1}_{k_{r+1}<k_r}}\\
&\sim n^{-r-1} K_r \sum_{k_1, \ldots, k_{r+1}} D^{(r+1)}\left( \frac{k_1}{n}, \ldots, \frac{k_{r+1}}{n} \right)\\
&\to K_r \int_{x_1, \ldots, x_{r+1}\in [0,1]} D^{(r+1)} (x_1, \ldots, x_{r+1}) \prod_{i=1}^{r+1} dx_i = K_r.
\end{align*}
This completes the proof of \eqref{eq:kr}, and thus of Proposition \ref{prop:loclim1}.
\end{proof}

\section{Mesoscopic limit of a nondegenerate Eulerian permutation}
\label{sec:mesoscopic}

In this section, we consider a sequence $(q_n)_{n \geq 1}$ with values in $(0,1)$ such that $q_n \to 0$ and $\log(1/q_n) = o(n)$, and we are interested in the asymptotic behaviour of the first $\lfloor c \log(1/q_n) \rfloor$ elements of $S_n^{(q_n)}$, for some fixed constant $c$ (observe that for constant $q_n$, this is already included in Theorem \ref{thm:mainconstant}). In this context, it turns out that a window of size $\Theta (\log(1/q_n))$ is the right choice to observe nontrivial asymptotic behaviour. In this window, it is natural to expect $S_n^{(q_n)}$ to behave like a degenerate permutation. Therefore, we first gather global results on degenerate permutations, before using them on nondegenerate ones.

\subsection{Global properties of a degenerate permutation}

In this part, fix $c>0$, and consider a sequence $(q_n)$ satisfying $\frac{n}{\log(1/q_n)} \underset{n \to \infty}{\to} c$. We investigate the asymptotic behaviour of the degenerate permutation $S_n^{(q_n)}$. 

%

%
%
%

Our result, in the same vein as Lemma \ref{lem:permufix}, shows that the study of a degenerate permutation boils down to investigating binomial variables. We consider the permutation $T_n^k(j)$ defined as follows. Let $A_1, \ldots, A_{j-1}$ be $j-1$ i.i.d.~random variables, all uniform on $\{ 2, \ldots, k+1 \}$, and let $A_{j+1}, \ldots, A_n$ be $n-j$ i.i.d.~random variables, all uniform on $\{ 1, \ldots, k+1 \}$. For convenience, set also $A_j=1$. For all $1 \leq i \leq k+1$, set $F_i := \{ p \in \llbracket 1, n \rrbracket, A_p=i \}$, and let $L_i$ be the list of all elements of $F_i$ sorted in increasing order. Finally, let $T_n^k(j)$ be the concatenation of $L_1, \ldots, L_{k+1}$.

\begin{lemma}
\label{lem:permufixfirstentryfixed}
Fix a nonnegative integer $k$, and let $E_n^k(j)$ be the event that $T_n^k(j)$ has $k$ descents. Then, uniformly for $j \in \llbracket 1,n \rrbracket$,
\begin{itemize}
\item[(i)] $\displaystyle \P \left( E_n^k(j) \right) \underset{n \to \infty}{\to} 1$.

\item[(ii)] 
Conditionally on $E_n^k(j)$, the permutation $T_n^k(j)$ is uniform among all permutations of $\kS_n$ with $k$ descents starting with $j$.
\end{itemize}
\end{lemma}

The proof of Lemma \ref{lem:permufixfirstentryfixed} is a direct adaptation of the proof of Lemma \ref{lem:permufix}. Different corollaries can be obtained from it.

\begin{cor}
\label{cor:degenerateperm}
\begin{itemize}
\item[(i)] For any fixed $k \geq 0$, uniformly in $j$, the number of permutations of $\kS_n$ with $k$ descents starting with $j$ is $k^{j-1} \, (k+1)^{n-j} (1+o(1))$.

\item[(ii)] For any $k$ and any $j$, the number of permutations of $\kS_n$ with $k$ descents starting with $j$ is at most $k^{j-1} \, (k+1)^{n-j}$.
\end{itemize}
\end{cor}

\begin{proof}
The proofs of (i) and (ii) are immediate: the number of possible outcomes in the construction of $T_n^k(j)$ is $k^{j-1} \, (k+1)^{n-j}$, so that everything follows from Lemma~\ref{lem:permufixfirstentryfixed} (i) and (ii). 
\end{proof}

The next result shows concentration of the number of descents, uniformly in the value of the first element of the permutation.

\begin{cor}
\label{cor:degenerateperm2}
Let  $D_n$ be the number of descents in $S_n^{(q_n)}$, where $\frac{n}{\log(1/q_n)} \underset{n \to \infty}{\to} c$.
For any $\epsilon>0$, for $n$ large enough, uniformly in $j \in \llbracket 1,n \rrbracket$,
\begin{align*}
\P \left( D_n \notin \llbracket \lfloor c \rfloor - 1, \lfloor c \rfloor + 1  \rrbracket \, \Huge| \, S_n^{(q_n)}(1)=j \right) < \epsilon.
\end{align*}
\end{cor}

\begin{proof}
Let us first prove that uniformly in $j$, for any $1 \leq k < \lfloor c \rfloor -1$, $\P(D_n^j = k) \to 0$ as $n \to \infty$, where $D_n^j$ denotes the number of descents in $S_n^{(q_n)}$ conditionally on $S_n^{(q_n)}(1)=j$. Observe that, by Corollary~\ref{cor:degenerateperm} (i), for fixed $k$, uniformly in $j$, the total weight of permutations of $\llbracket 1, n \rrbracket$ starting at $j$ with $k$ descents is
\begin{align*}
k^{j-1} \, (k+1)^{n-j} q_n^k (1+o(1)) &= \exp \big( (j-1) \log k + (n-j) \log(k+1) + k \log q_n + o(1) \big)\\
&= \exp \left( (j-1) \log \left(\frac{k}{k+1}\right) + (n-1) \log(k+1) - k n/c \left( 1 + o(1) \right) \right)\\
&= \exp \left(n \left( \frac{j}{n} \log \left(\frac{k}{k+1}\right) + \log(k+1) - \frac{k}{c} \right) + o(n) \right).
\end{align*}
Let $f(x,k) = x \log(k/(k+1)) + \log(k+1) - k/c$. Then, for all $x \in [0,1]$,
\begin{equation}
\label{eq:boundsonf}
\frac{\partial f}{\partial k} (x,k) \in \left[ \frac{1}{k+1}- \frac{1}{c}, \frac{1}{k}- \frac{1}{c} \right]
\end{equation}
so that $\argmax_k f(x,k) \in \llbracket \lfloor c \rfloor - 1, \lfloor c \rfloor + 1 \rrbracket$. Since the bounds in \eqref{eq:boundsonf} do not depend on $x$, we obtain indeed that, uniformly in $j$, for any $1 \leq k < \lfloor c \rfloor -1$, $\P(D_n^j = k) \to 0$ as $n \to \infty$.

Now let us prove that $\sup_{1 \leq j \leq n} \P \left( D_n \geq \lfloor c \rfloor + 2 \, \Huge| \, S_n^{(q_n)}=j \right) \to 0$. To this end, observe that for any $j$, by Corollary \ref{cor:degenerateperm} (ii), we have that the weighted sum of permutations of $\llbracket 1, n \rrbracket$ starting with $j$ and with at least $\lfloor c \rfloor + 2$ descents is at most

\begin{align*}
\sum_{k \geq \lfloor c \rfloor + 2} k^{j-1} (k+1)^{n-j} q_n^k &= \sum_{k \geq \lfloor c \rfloor + 2} \exp \big( h(j,k,n) \big),
\end{align*}
where $h(j,k,n) = (j-1) \log k + (n-j) \log(k+1) + k \log q_n$. Now we use the fact that, for all $j,n$,
\begin{align*}
\frac{\partial h}{\partial k}(j,k,n) = \frac{j-1}{k} + \frac{n-j}{k+1} + \log q_n \leq \frac{n}{k} + \log q_n.
\end{align*}

By definition of $q_n$, there exists $\epsilon(c)>0$ depending only on $c$ such that, for $n$ large enough and any $k \geq \lfloor c \rfloor +2$, for any $1 \leq j \leq n$, $\frac{\partial h}{\partial k}(j,k,n) < - n \epsilon(c)$. Therefore,
\begin{align*}
\sum_{k \geq \lfloor c \rfloor + 2} \exp \big( h(j,k,n) \big) &\leq \exp \big( h(j,\lfloor  c \rfloor + 2, n) \big) \sum_{k \geq 0} \exp \big({-} n \epsilon(c) k \big) \\
&= \exp \big(h(j,\lfloor  c \rfloor + 2,n)\big) \frac{1}{1-e^{-n\epsilon(c)}}.
\end{align*}
By the same argument as before, we get that $\exp \left(h(j,\lfloor  c \rfloor + 2,n)\right) = o \left( \exp \left(h(j,\lfloor  c \rfloor + 1,n)\right) \right)$, and the result follows.
\end{proof}

\subsection{Mesoscopic limit of a nondegenerate permutation}

We now apply these results to the study of a nondegenerate permutation. In what follows, we consider a sequence $(q_n)_{n \geq 1}$ satisfying $q_n \to 0$ and $\log q_n = o(n)$ as $n \to \infty$. 

\paragraph*{Number of descents in the first $\lfloor c \log(1/q_n) \rfloor$ elements.}

Fix $c>0$. The following holds by Corollary \ref{cor:degenerateperm2} and a time-reversing argument.

\begin{lemma}
\label{lem:numberofdescents}
Let  $D(x)$ denote the number of descents in the first $\lfloor x \rfloor$ elements of $S_n^{(q_n)}$. Then, with high probability,
\begin{align*}
D \left(c \log(1/q_n) \right) \in \llbracket \lfloor c \rfloor -1, \lfloor c \rfloor +1 \rrbracket.
\end{align*}
\end{lemma}

\begin{proof}[Proof of Lemma \ref{lem:numberofdescents}]
Let $E_c$ denote the event that $D \left(c \log(1/q_n) \right) \in \llbracket \lfloor c \rfloor -1, \lfloor c \rfloor +1 \rrbracket$ and, for any $1 \leq j \leq \lfloor c \log(1/q_n) \rfloor$, let $F_c^j$ be the event that $S_n^{(q_n)}(\lfloor c \log(1/q_n)\rfloor)$ is the $j$-th largest element among the first $\lfloor c \log(1/q_n)\rfloor$.
By the law of total probability, we have
\begin{align*}
\P\left( E_c \right) = \sum_{j=1}^{\lfloor c \log(1/q_n) \rfloor} \P \left( E_c \huge| F_c^j \right) \P \left( F_c^j \right).
\end{align*}

Now, consider the permutation $T_n \in \kS_{\lfloor c \log(1/q_n) \rfloor}$, defined as follows: for all $1 \leq \ell \leq \lfloor c \log(1/q_n) \rfloor$, $T_n(\ell)=m$ if $S_n^{(q_n)}(\lfloor c \log(1/q_n) \rfloor-\ell+1)$ is the $m$-th largest element among the first $\lfloor c \log(1/q_n) \rfloor$ of $S_n^{(q_n)}$. Then, for any $j$, conditionally on $F_c^j$, $T_n$ is distributed as $S_{\lfloor c \log(1/q_n) \rfloor}^{(q_n)}$ conditioned on $S_{\lfloor c \log(1/q_n) \rfloor}^{(q_n)}(1)= j$. Lemma \ref{lem:numberofdescents} follows by Corollary \ref{cor:degenerateperm2}, since $T_n$ and $S_n^{(q_n)}\left(\llbracket 1, \lfloor c \log(1/q_n) \rfloor \rrbracket \right)$ have the same number of descents.
\end{proof}

\paragraph*{Localization of the descents.}

For all $i \geq 1$, let $I_i$ denote the location of the $i$-th descent in $S_n^{(q_n)}$, that is, $I_1 := \inf \{ k \geq 1, S_n^{(q_n)}(k)> S_n^{(q_n)}(k+1) \}$ and $I_{i+1} := \inf \{ k \geq I_i+1, S_n^{(q_n)}(k)> S_n^{(q_n)}(k+1) \}$.
It turns out that these descents are localized, in the following sense:

\begin{lemma}
\label{lem:localizationfirstjump}
For all $i \geq 1$, we have the following convergence in probability:
\begin{align*}
\frac{I_i}{\log(1/q_n)} \underset{n \to \infty}{\overset{\P}{\to}} i.
\end{align*}
\end{lemma}

\begin{proof}[Proof of Lemma \ref{lem:localizationfirstjump}]
We know from Lemma \ref{lem:numberofdescents} that, with high probability, the number of descents in $S_n^{(q_n)}(\llbracket 1, \lfloor c\log(1/q_n) \rfloor \rrbracket)$ is in $\llbracket \lfloor c \rfloor -1, \lfloor c \rfloor +1 \rrbracket $. Now, we condition on this number of descents $k$. Let $B_k^c$, $B_{k+1}^c$ be two random variables such that $B_k^c \sim \Bin(\lfloor c \log(1/q_n)\rfloor,1/k)$ and $B_{k+1}^c \sim \Bin(\lfloor c \log(1/q_n)\rfloor,1/(k+1))$. Letting $\tilde{E}_j^k$ be the event that $S_n^{(q_n)}(\lfloor c \log(1/q_n) \rfloor)$ is the $j$-th largest among the first $\lfloor c \log (1/q_n) \rfloor$ elements, we obtain by Lemma \ref{lem:permufixfirstentryfixed} that, for all $j$, conditionally on $\tilde{E}_j^k$, there exists a coupling between $B_k^c, I_1$ and $B_{k+1}^c$ such that, with high probability as $n \to \infty$ (uniformly in $j$):

\begin{equation*}
B_{k+1}^c \leq I_1 \leq B_k^c.
\end{equation*} 

In particular, for any $\epsilon > 0$, taking $c$ large enough, we get that with high probability,

\begin{align*}
\left|\frac{I_1}{\log(1/q_n)} - 1\right| < \epsilon.
\end{align*}
This implies the result for $I_1$, and the result for each fixed $i$ is proved in the same way.
The result follows.
\end{proof}

\paragraph*{Distribution of the first elements of $S_n^{(q_n)}$.}

We are finally able to prove the main result in this section. 
%
%
%
%
%
%
In what follows, we make no distinction between a permutation $\sigma \in \kS_n$ and the c\`{a}dl\`{a}g function $f_\sigma$ defined as
\begin{align*}
f_\sigma(x) = \sigma(1 + \lfloor x \rfloor) \text{ for all } x \in [0,n).
\end{align*}

We also endow the set of càdlàg functions on an interval with the usual $J_1$ Skorokhod topology (see \cite[Chapter $VI$]{JS03} for details).

%

%

\begin{theorem}
\label{thm:distribution}
We have, with respect to the Skorokhod topology, for any non-integer $c \in \R_+$:
\begin{align*}
\left( \frac{S_n^{(q_n)}(\lfloor t \log(1/q_n) \rfloor +1)}{n} \right)_{0 \leq t \leq c} \underset{n \to \infty}{\overset{(d)}{\to}} \left( \{ t \} \right)_{0 \leq t \leq c}.
\end{align*}
\end{theorem}

The proof uses a sort of renewal argument.

\begin{proof}[Proof of Theorem \ref{thm:distribution}]
Let us first prove the result for $c<1$. Let $\tilde{S}_n^{(q_n)}$ denote a permutation with the same distribution as $S_n^{(q_n)}$, and independent of it.  Conditionally on $I_1$, it is clear that $S_n^{(q_n)}(\llbracket I_1+1,n\rrbracket)$ is distributed as $\tilde{S}_{n-I_1}^{(q_n)}$ (up to increasing bijection from the range of $S_n^{(q_n)}(\llbracket I_1+1,n\rrbracket)$ to $\llbracket 1, n-I_1 \rrbracket$), under the event that $S_n^{(q_n)}(I_1)>S_n^{(q_n)}(I_1+1)$.

We condition on the value of $I_1$, and on the set of values of the first $I_1$ elements of $S_n^{(q_n)}$.  Fix $\varepsilon>0$. By Theorem \ref{thm:mainnondegenerate} and Lemma \ref{lem:localizationfirstjump}, we know that, with high probability, $S_n^{(q_n)}(I_1) \gg \frac{n}{\log(1/q_n)}$ and that there exists $C>0$ such that $\P(\tilde{S}_{n-I_1}^{(q_n)}(1)\leq C\frac{n}{\log(1/q_n)}) <\varepsilon$. Therefore, with high probability, conditionally on $I_1$, the range of $S_n^{(q_n)}(\llbracket 1, I_1 \rrbracket)$ is asymptotically distributed as a uniform set of $I_1$ elements of $\llbracket 1, n \rrbracket$. Furthermore, by Lemma \ref{lem:localizationfirstjump}, with high probability $I_1 > c \log(1/q_n)$. The result follows, and can be extended in the same way to any non-integer value of $c$.
\end{proof}

\section{Generating functions for descent-biased trees}\label{sec:genfun_trees}

In this and the following section, we study the model of descent-biased random trees $\cT_n^{(q)}$ described in the introduction. First, we are interested in generating functions associated with this model. The enumeration of labelled trees by descents and the study of associated generating functions goes back to a paper of E{\u{g}}ecio{\u{g}}lu and Remmel \cite{ER86}, where the number of descents is considered jointly with other statistics.

We will use generating functions to determine the limiting distribution of the root degree for every fixed $q$, which is the first step in characterizing the local limit. Moreover, we determine the asymptotic behaviour of the mean of the path length, which is the sum of all distances to the root. While interesting in its own right, this is also relevant as an auxiliary tool for the proof of Lemma~\ref{lem:heightn}, where we obtain bounds on the height of the tree $\cT_n^{(q)}$.  The information that we gain from the analysis of generating functions will be combined with our results on descent-biased permutations to prove a local limit theorem in the following section. 

Recall that we consider the set $\kT_n$ of rooted trees (that is, trees with a marked vertex), non-plane, with vertices labelled from $1$ to $n$, where $n$ denotes the number of vertices in the tree. An edge is a \textit{descent} if the parent label is greater than the child label. For any tree $T$, we write $d(T)$ for the number of descents in $T$. Recall also that, for $q > 0$ and $n \geq 1$, the random tree $\cT_n^{(q)}$ is the random variable on $\kT_n$ satisfying, for any $T \in \kT_n$,
\begin{align*}
\P \left( \cT_n^{(q)} = T \right) = \frac{1}{Y_{n,q}} q^{d(T)},
\end{align*}
where $Y_{n,q} := \sum_{T \in \kT_n} q^{d(T)}$.

\subsection{The ordinary generating function}

 Let us first define the generating function of labelled trees with a weight $q^{d(T)}$, i.e.,
$$A(x,q) = \sum_{n \geq 1} \sum_{T \in \kT_n} q^{d(T)} \frac{x^n}{n!}.$$
We decompose a labelled tree $T$ according to the vertex labelled $1$: if this vertex is the root, then the rest can be regarded as a set of smaller labelled trees $T_1,T_2,\ldots,T_k$, with the number of descents in the different subtrees adding up to the number of descents in $T$ (the root does not induce any additional descents in this case). Otherwise, we can decompose $T$ (with vertex $1$ removed) into a part $R$ containing the root and the parent of vertex $1$ (which may be identical) on the one hand, and the subtrees $T_1,T_2,\ldots,T_k$ rooted at the children of vertex $1$ on the other hand. The number of descents in $T$ is obtained by summing the number of descents over all these parts, and adding $1$ for the descent induced by vertex $1$ and its parent. Combining these two cases leads to the generating function identity
\begin{equation}\label{eq:A-diffeq}
\frac{\partial}{\partial x} A(x,q) = e^{A(x,q)} + qx e^{A(x,q)} \frac{\partial}{\partial x} A(x,q),
\end{equation}
which can be rewritten as
\begin{equation}\label{eq:A_x-expr}
\frac{\partial}{\partial x} A(x,q) = \frac{e^{A(x,q)}}{1-qx e^{A(x,q)}}.
\end{equation}
We solve this differential equation by first setting $f(x) = xe^{A(x,q)}$ (ignoring the dependence on $q$ for the moment). Since
$$f'(x) = e^{A(x,q)} + x e^{A(x,q)} \frac{\partial}{\partial x} A(x,q) = \frac{f(x)}{x} \Big( 1 + x \frac{\partial}{\partial x} A(x,q) \Big),$$
this gives us
$$f'(x) = \frac{f(x)}{x} \Big( 1 + \frac{f(x)}{1-qf(x)} \Big).$$
This differential equation is separable and can be solved by standard methods. One obtains the implicit equation
$$f(x) (1 + (1-q)f(x))^{-1/(1-q)} = Cx$$
for some constant $C$. Writing $f(x) = xe^{A(x,q)}$ and simplifying (as well as using the fact that $A(0,q) = 0$ to determine $C$), this eventually leads to
\begin{equation}\label{eq:impl_eq_A}
x = \frac{e^{-qA(x,q)} - e^{-A(x,q)}}{1-q}.
\end{equation}
Note that it is possible to determine the explicit formula
$$[x^n] A(x,q) = \frac{1}{n!} \prod_{k=1}^{n-1} (n-k+kq)$$
(which also follows by specialization from the results in \cite{ER86}) from this implicit equation by means of Lagrange inversion, but this will not be relevant to us. We are rather interested in the singularities of $A(x,q)$ for fixed $q > 0$ in order to be able to apply singularity analysis to $A(x,q)$ and related generating functions. Note that $x \mapsto A(x,q)$ is the inverse function of the entire function $a \mapsto F(a) = \frac{e^{-qa}-e^{-a}}{1-q}$ (if $q = 1$, we take the limit $F(a) = a e^{-a}$) that satisfies $A(0,q) = 0$. By the analytic implicit function theorem, it can be continued analytically at every complex $x$ as long as the derivative of $F$ does not vanish at $A(x,q)$. This derivative is
$$F'(a) = \frac{1}{1-q} \Big( e^{-a} - q e^{-qa} \Big)$$
(if $q=1$, it is $(1-a)e^{-a}$), which vanishes at points of the form $\frac{\log q}{q-1} + \frac{2k \pi i}{q-1}$ (or only at $1$ if $q = 1$). So we reach a singularity when $A(x,q) =  \frac{\log q}{q-1}$ and accordingly $x = q^{q/(1-q)}$ (for $q=1$, we need to have $A(x,1) = 1$ and $x = e^{-1}$ respectively). Let us denote this singularity by $x_0$: $x_0:= q^{q/(1-q)}$ (if $q = 1$, this is interpreted as $e^{-1}$). This is in fact the dominant singularity: at any positive real $x < x_0$, $A(x,q)$ can be continued analytically by the analytic implicit function theorem. By Pringsheim's theorem (as $A(x,q)$ has only nonnegative coefficients), this already means that it is analytic in the open disk of radius $x_0$ around $0$. For every complex $x$ with $|x| = x_0$ but $x \neq x_0$, we have $|A(x,q)| < A(x_0,q) = \frac{\log q}{q-1}$ by the triangle inequality, thus in particular $A(x,q) \neq \frac{\log q}{q-1} + \frac{2k \pi i}{1-q}$ for all integers $k$. So $A(x,q)$ can also be continued analytically at all these points.

By standard results on implicitly defined functions (see \cite[Chapter VI.7]{FS09}), $A(x,q)$ is amenable to singularity analysis with a square root singularity at $x_0$:

\begin{lemma}\label{lem:analytic_properties_A}
For every fixed $q > 0$, there exists $\delta > 0$ such that $A(x,q)$ is analytic in a domain of the form
$$\{x \in \C \,:\, |x| < x_0 + \delta, x \notin[x_0,x_0+\delta)\},$$
and $A(x,q)$ has the following asymptotic expansion at $x_0$:
$$A(x,q) = \frac{\log q}{q-1} - \sqrt{\frac{2}{q}} \cdot \sqrt{1-x/x_0} + O \big( |x-x_0| \big).$$
\end{lemma}

Consequently, singularity analysis \cite[Theorem VI.4]{FS09} yields
\begin{equation}\label{eq:Axq_coeff_asymp}
[x^n] A(x,q) \sim \frac{1}{\sqrt{2\pi q}} n^{-3/2} x_0^{-n}
\end{equation}
for every fixed $q > 0$. For our purposes, we will need a more general statement involving powers of $A(x,q)$. This is provided in the following proposition.

\begin{proposition}\label{prop:power_coefficients}
Let $q \in (0,1)$ be fixed. There exists an analytic function $h$ in a neighbourhood of $0$ with $h(0) = 0$ such that the coefficient of $x^n$ in $A(x,q)^m$ is asymptotically given by
$$[x^n] A(x,q)^m \sim \frac{m(1-q)^{n+1-m}(\log \frac1{q})^{m-1} e^{-m h(m/n)}(q^{qe^{-h(m/n)}/(1-q)} - q^{e^{-h(m/n)}/(1-q)})^{-n}}{n^{3/2}\sqrt{2\pi q}}$$
provided that $m = o(n)$. This holds uniformly in $m$ for all $m \in \llbracket 1, M(n) \rrbracket$ if $M(n) = o(n)$.

As a consequence, if $n \to \infty$ and $m = o(n)$, we have 
$$\frac{[x^{n+1}] A(x,q)^m}{[x^{n}] A(x,q)^m} \to q^{-q/(1-q)},$$
and more generally 
$$\frac{[x^{n+k}] A(x,q)^m}{[x^{n}] A(x,q)^m} \to q^{-qk/(1-q)}$$
for every fixed $k$. Moreover, if $n,m \to \infty$ with $m = o(n)$, we have
$$\frac{[x^{n}] A(x,q)^{m-1}}{[x^{n}] A(x,q)^m} \to \frac{1-q}{\log(1/q)}.$$
Again, these statements hold uniformly in $m$ if $m \in \llbracket M_1(n), M_2(n) \rrbracket$, where $M_1(n) \to \infty$ and $M_2(n) = o(n)$.
\end{proposition}


\begin{proof}
As a first step, we express the coefficient $[x^n] A(x,q)^m$ as a complex contour integral. Note first that $[x^n] A(x,q)^m = \frac{1}{n} [x^{n-1}] \frac{d}{dx} A(x,q)^m = \frac{m}{n} [x^{n-1}] A(x,q)^{m-1} \frac{d}{dx}A(x,q)$ (writing the coefficient in this way is motivated by the Lagrange inversion formula, see \cite[Theorem A.2]{FS09}). By Cauchy's integral formula, we have
$$[x^n] A(x,q)^m = \frac{m}{n} \cdot \frac{1}{2\pi i} \oint_{\cC} \frac{A(x,q)^{m-1} \frac{d}{dx}A(x,q)}{x^n} \,dx,$$
where the integral is taken along a closed contour $\cC$ surrounding $0$ (to be specified later) such that $A(x,q)$ is analytic inside. Now perform the change of variables $A(x,q) = z$ and express $x$ in terms of $z$ by means of~\eqref{eq:impl_eq_A} to obtain
\begin{equation}\label{eq:contour_integral}
[x^n] A(x,q)^m = \frac{m(1-q)^n}{n} \cdot \frac{1}{2\pi i} \oint_{\cC'} \frac{z^{m-1}}{(e^{-qz}-e^{-z})^n}\,dz.
\end{equation}
Let us now apply the saddle point method to this integral. We first identify the saddle point: set $f(z) = \log (e^{-qz} - e^{-z})$. We choose $z_0$ in such a way that $m = n z_0 f'(z_0)$. Note here that $f'(a_0) = 0$ for $a_0 = \frac{\log q}{q-1}$. Moreover, we have the Taylor expansion $z f'(z) = \frac{q \log q}{1-q} (z-a_0) + O( (z-a_0)^2 )$, so the function $z \mapsto z f'(z)$ is locally invertible around $a_0$. Thus there exists an increasing analytic function $h$ in some interval around $0$ such that $h(0) = 0$ and $w \mapsto a_0 e^{-h(w)}$ is the inverse function to $z \mapsto z f'(z)$. If $m = o(n)$, then for sufficiently large $n$ the quotient $m/n$ lies in this interval, so that $z_0 = a_0 e^{-h(m/n)}$ is a solution to the equation $m = n z_0 f'(z_0)$.

Next, we choose the integration contour in such a way that it passes through $z_0$. Specifically, we will define $\cC'$ in such a way that it contains the circular arc $\cA = \{z_0 e^{it} \,:\, |t| < \eta\}$ for a suitable constant $\eta > 0$. Note first that since $f(z)$ is increasing on the interval $[0,a_0]$, we have
$$(1-q)^{-1} e^{f(z_0)} < (1-q)^{-1} e^{f(a_0)} = q^{q/(1-q)} = x_0.$$
Thus the image of $z_0$ under the map $z \mapsto (1-q)^{-1} e^{f(z)}$ lies inside the region of analyticity of $A(x,q)$. Moreover, the function $t \mapsto (1-q)^{-1}e^{f(z_0 e^{it})}$ has the Taylor expansion (writing $h = h(m/n)$)
\begin{multline*}
\big( x_0 - O(h) \big) + O(h) it + \Big( \frac{q^{1/(1-q)} \log^2 q}{2(1-q)^2} + O(h) \Big) t^2 \\
+ \Big( \frac{q^{1/(1-q)}\log^2 q (3(1-q)+(1+q)\log q)}{6(1-q)^3} + O(h) \Big) i t^3 + O(t^4).
\end{multline*}
Let $\eta > 0$ be small, but fixed, and let us denote the image of the circular arc $\cA = \{z_0 e^{it} \,:\, |t| < \eta\}$ under the map $z \mapsto (1-q)^{-1} e^{f(z)}$ by $\tilde{\cA}$.
In view of the Taylor expansion above, if $\eta$ is chosen suitably and $h$ is small enough (recall that we are assuming $m = o(n)$, thus $h = h(m/n) = o(1)$), then $\tilde{\cA}$ lies within the domain
$$\{x \in \C \,:\, |x| < x_0 + \delta, x \notin[x_0,x_0+\delta)\},$$
where $A(x,q)$ is analytic. Moreover, the modulus is increasing in $|t|$, and the endpoints $x_+$ and $x_-$ of $\tilde{\cA}$ have modulus greater than or equal to $x_0 (1 + \epsilon)$ for some fixed $\epsilon > 0$  if $h$ is small enough. See Figure~\ref{fig:contour}.

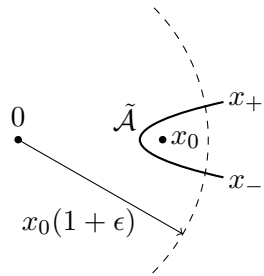
\begin{figure}[htbp]
\centering
\begin{tikzpicture}

\node[fill=black,circle,inner sep=1pt]  at (0,0) {};
\node at (0,0.3) {$0$};

\node[fill=black,circle,inner sep=1pt]  at (1.9,0) {};
\node at (2.2,0) {$x_0$};

\draw [dashed] (1.76776695297,-1.76776695297) arc (-45:45:2.5);

\draw [->] (0,0) -- (2.16506350946,-1.25);

\node at (0.8,-1.1) {$x_0(1+\epsilon)$};

\draw[rotate=-90,thick] (-0.5,2.7) parabola bend (0,1.6) (0.5,2.7);



\node at (1.4,0.3) {$\tilde{\cA}$};
\node at (3,-0.6) {$x_-$};
\node at (3,0.5) {$x_+$};

\end{tikzpicture}
\caption{Constructing the contour: the image $\tilde{\cA}$ of $\cA$.}\label{fig:contour}
\end{figure}

If we add the circular arc between $x_+$ and $x_-$ that is centred at $0$ and does not intersect the positive real axis, we obtain a closed curve $\cC$. Its image under the map $x \mapsto A(x,q)$ is a closed curve around $0$  that we call $\cC'$ (equivalently, the preimage of $\cC'$ under the map $z \mapsto (1-q)^{-1} e^{f(z)}$ is $\cC$ in view of~\eqref{eq:impl_eq_A}). This is the contour that we use to evaluate the integral~\eqref{eq:contour_integral}.

By construction, $\cC'$ contains the circular arc $\cA = \{z_0 e^{it} \,:\, |t| < \eta\}$.  For all $z$ on the rest of $\cC'$, we know that $|\frac{e^{-qz}-e^{-z}}{1-q}| = |(1-q)^{-1} e^{f(z)}| = |x_+| = |x_-| \geq x_0(1+\epsilon)$ by construction. Therefore, the contribution of $\cC' \setminus \cA$ to the integral in~\eqref{eq:contour_integral} is bounded above by $x_0^{-n} (1+\epsilon)^{-n} e^{o(n)}$ by the assumption that $m = o(n)$. This turns out to be negligible.

For the remaining integral along $\cA$, we can parametrize $z = z_0 e^{it}$ with $|t| < \eta$. This yields
\begin{equation}\label{eq:saddle_integral}
\frac{m (1-q)^n}{n} \cdot \frac{1}{2\pi i} \oint_{\cA} \frac{z^{m-1}}{(e^{-qz}-e^{-z})^n}\,dz = \frac{m(1-q)^n z_0^m}{n} \cdot \frac{1}{2\pi} \int_{-\eta}^{\eta} \exp \big( imt - n f(z_0 e^{it}) \big) \,dt.
\end{equation}
By our choice of $z_0$, the linear term in the Taylor expansion of the exponent vanishes: 
$$\frac{d}{dt} (i m t - n f(z_0e^{it})) \big|_{t=0} = i m - inz_0 f'(z_0) = 0.$$
 So we have
$$imt - n f(z_0 e^{it}) = -nf(z_0) - \Big( \frac{q \log^2 q}{2(1-q)^2} + O(h) \Big) n t^2 + O(n t^3).$$
The integral in~\eqref{eq:saddle_integral} can now be evaluated asymptotically by standard methods (see \cite[Chapter VIII]{FS09}): it is equal to
$$e^{-n f(z_0)} \Big(\frac{2\pi (1-q)^2}{q \log^2 q} + O(h) \Big)^{1/2} n^{-1/2} (1 + o(1)).$$
Thus~\eqref{eq:saddle_integral} evaluates to
$$\frac{m(1-q)^n e^{-n f(z_0)} z_0^m}{n^{3/2}} \cdot \Big( \frac{(1-q)^2}{2\pi  q \log^2 q} \Big)^{1/2} (1 + o(1)),$$
taking also into account that $h = o(1)$. Since $(1-q)^{-1}e^{f(z_0)} = (1-q)^{-1}e^{f(a_0) + o(1)} = x_0 (1+o(1))$, this also shows that the contribution from the part of $\cC'$ outside the circular arc $\cA$ is negligible as claimed. Putting everything together, the statement of the proposition follows.
\end{proof}

\begin{remark}
For $q = 1$, we have the explicit formula $[x^n] A(x,1)^m = \frac{m n^{n-m-1}}{(n-m)!}$ (which can be proven by Lagrange inversion), for which the second part of the statement of Proposition~\ref{prop:power_coefficients} is easily verified as well: one has
$$\frac{[x^{n+k}] A(x,1)^m}{[x^{n}] A(x,1)^m} \sim e^k$$
for every fixed $k$ as $n \to \infty$ and $m = o(n)$, and
$$\frac{[x^{n}] A(x,1)^{m-1}}{[x^{n}] A(x,1)^m} \to 1$$
if $n,m \to \infty$ and $m = o(n)$.
\end{remark}

\subsection{Root degree}

Let us now incorporate the root degree into the generating function $A(x,q)$: let $\rd(T)$ denote the root degree of a tree $T$, and set
$$A(x,q,u) = \sum_{n \geq 1} \sum_{T \in \kT_n} q^{d(T)} u^{\rd(T)} x^n.$$
The differential equation~\eqref{eq:A-diffeq} can be modified easily to include the new variable $u$: we have
$$
\frac{\partial}{\partial x} A(x,q,u) = e^{u A(x,q)} + qx e^{A(x,q)} \frac{\partial}{\partial x} A(x,q,u) + (u-1)q A(x,q,u) e^{A(x,q)},
$$
since the root degree is either the degree of vertex $1$ (if it is the root), or else the degree of the root in the root component after removal of vertex $1$, with $1$ added if the root is also the parent of vertex $1$. Note that this is a linear differential equation for the generating function $A(x,q,u)$ ($A(x,q) = A(x,q,1)$ can already be assumed to be known) that can be solved: set $f(x) =  A(x,q,u)$ for the time being, and rewrite the equation as
$$f'(x) +(1-u)q \cdot \frac{e^{A(x,q)}}{1-qx e^{A(x,q)}} f(x) = \frac{e^{u A(x,q)}}{1-qx e^{A(x,q)}}.$$
By virtue of~\eqref{eq:A_x-expr}, this can be simplified to (writing $A_x(x,q)$ for the partial derivative of $A(x,q)$ with respect to $x$)
$$f'(x) +(1-u)q A_x(x,q) f(x) = e^{(u-1)A(x,q)} A_x(x,q).$$
Multiply by the integrating factor $e^{(1-u)q A(x,q)}$ and integrate both sides to obtain
$$e^{(1-u)q A(x,q)} f(x) = C - \frac{1}{(u-1)(q-1)} e^{(1-q)(u-1)A(x,q)}$$
for a constant $C$. For $x = 0$, we must have $f(0) = A(0,q) = 0$, so $C$ is in fact $\frac{1}{(u-1)(q-1)}$. After some further simplification, we end up with
\begin{equation}\label{eq:Axqu-explicit}
A(x,q,u) = f(x) = \frac{1}{(u-1)(q-1)} \Big( e^{q(u-1)A(x,q)} - e^{(u-1)A(x,q)} \Big).
\end{equation}
The behaviour at the dominant singularity $x_0$ (which $A(x,q,u)$ inherits from $A(x,q)$) is particularly relevant for us: from Lemma~\ref{lem:analytic_properties_A}, we obtain
$$A(x,q,u) = \frac{1}{(u-1)(q-1)} \Big( q^{\frac{q(u-1)}{q-1}} - q^{\frac{u-1}{q-1}} \Big) - \frac{1}{q-1} \sqrt{\frac{2}{q}} \Big( q^{\frac{qu-1}{q-1}} - q^{\frac{u-1}{q-1}} \Big) \sqrt{1-x/x_0} + O \big( |x-x_0| \big)$$
for every fixed $u$ (and fixed $q$). Applying singularity analysis to this expansion in the same way as for~\eqref{eq:Axq_coeff_asymp}, we obtain
$$\frac{[x^n] A(x,q,u)}{[x^n] A(x,q)} \to \frac{q^{\frac{qu-1}{q-1}} - q^{\frac{u-1}{q-1}}}{q-1} =: g(u),$$
as $n \to \infty$. The fraction on the left is precisely the probability generating function of the root degree of $\cT_n^{(q)}$. In addition, $g(1)=1$ and the coefficients $[u^k] g(u)$ are nonnegative for all $k \geq 0$. Since the probability generating functions converge to a probability generating function, we immediately obtain convergence of the root degree distribution to a discrete limiting distribution, see \cite[Theorem IX.1]{FS09}:

\begin{proposition}
For fixed $q > 0$, the limiting distribution of the root degree in descent-biased trees $\cT_n^{(q)}$ as $n \to \infty$ is given by the probability generating function
$$p(u) = \frac{q^{\frac{qu-1}{q-1}} - q^{\frac{u-1}{q-1}}}{q-1}.$$
Thus, for every fixed $k \geq 0$, the probability that the root degree of $\cT_n^{(q)}$ is equal to $k$ tends to $\frac{q^{1/(1-q)}}{q-1} \Big(\frac{\log q}{q-1} \Big)^k \frac{q^k-1}{k!}$.
\end{proposition}

As always, we only need to take the limit as $q \to 1$ to obtain the familiar result for the root degree of unbiased rooted labelled trees: in this case, the limiting probability generating function is $ue^{u-1}$, and the probability that the root degree is equal to $k$ tends to $\frac{1}{e (k-1)!}$ (see \cite[Example IX.6]{FS09}).

\medskip

Let us also consider the expected value of the root degree. Taking the derivative with respect to $u$ and letting $u \to 1$ in~\eqref{eq:Axqu-explicit}, we obtain
\begin{align*}
\frac{\partial}{\partial u} A(x,q,u) \Big|_{u=1} &= \frac{q+1}{2} A(x,q)^2 \\ 
&= \frac{(1+q)(\log q)^2}{2(q-1)^2} - \frac{\sqrt{2}(1+q)\log q}{\sqrt{q}(q-1)} \cdot \sqrt{1-x/x_0} + O \big( |x-x_0| \big).
\end{align*}
Applying singularity analysis, we thus find that the expected value of the root degree of $\cT_n^{(q)}$ is
\begin{equation}\label{eq:mean_root_deg}
\frac{[x^n] \frac{\partial}{\partial u} A(x,q,u) |_{u=1}}{[x^n] A(x,q)} \underset{n \to \infty}{\to}  \frac{(1+q)\log q}{q-1}.
\end{equation}
For $q=1$, the limit is $2$. In particular, the expected root degree is bounded. This can be generalized further:

\begin{proposition}\label{prop-r-radius-bounded}
For fixed $q > 0$ and every fixed positive integer $r$, the expected number of vertices whose distance from the root of $\cT_n^{(q)}$ is at most $r$ is bounded as $n \to \infty$.
\end{proposition}

\begin{proof}
The case $r=1$ is given by~\eqref{eq:mean_root_deg}. Consider any of the subtrees rooted at a child of the root, and condition on its size $k$. Compared to $\cT_k^{(q)}$, the probability of every given tree in $\kT_k$ (where we consider labels only up to order isomorphism) can only change by at most a factor of $q$ due to the potential descent between the root and its child. Consequently, the expected value of the root degree of the subtree can only change by at most a factor of $q$, and is in particular uniformly bounded by a constant. The statement follows immediately for $r=2$, and in a similar way for all values of $r$ by induction.
\end{proof}

\subsection{Path length}

In this subsection, we determine the asymptotic behaviour of the mean path length, i.e., the average sum of the distances from the root to all other vertices. 
Let $\PL(T)$ be the path length of a tree $T$, and let us define the generating function 
$$B(x,q) = \sum_{n \geq 1} \sum_{T \in \kT_n} q^{d(T)} \PL(T) \frac{x^n}{n!},$$
which incorporates the path length. Again, we can use the decomposition that gave us~\eqref{eq:A-diffeq}: if vertex $1$ is the root, then we notice that the path length of $T$ is the sum of the path lengths in all subtrees rooted at the root's children, plus $1$ for each non-root vertex. Otherwise, when $T$ is decomposed into a part $R$ containing the root, vertex $1$ and the branches $T_1,T_2,\ldots,T_k$ rooted at the children of vertex $1$, we obtain $\PL(T)$ by
\begin{itemize}
\item adding $\PL(R),\PL(T_1),\PL(T_2),\ldots,\PL(T_k)$,
\item adding the distance $d$ of $1$ from the root, and
\item for every vertex in $T_1,T_2,\ldots,T_k$, adding another $d+1$.
\end{itemize}
This eventually translates to the generating function identity (writing $A_x(x,q)$ for the partial derivative with respect to $x$ as before)
\begin{align*}
\frac{\partial}{\partial x} B(x,q) &= \big( x A_x(x,q) + B(x,q) \big) e^{A(x,q)} + qx e^{A(x,q)} \frac{\partial}{\partial x} B(x,q) \\
&\qquad +  qx A_x(x,q) \big( x A_x(x,q) + B(x,q) \big) e^{A(x,q)} \\
&\qquad + q \big( x A_x(x,q) + B(x,q) \big) \big(1 + x A_x(x,q)\big) e^{A(x,q)}.
\end{align*}
Again, this is a linear differential equation for $B(x,q)$ that can be solved by the same method as in our analysis of the root degree. Set $f(x) = B(x,q)$, and write the differential equation as
$$
f'(x) - \frac{(1+q+2q x A_x(x,q))e^{A(x,q)}}{1-qx e^{A(x,q)}} f(x) = \frac{x A_x(x,q) (1+q+2q x A_x(x,q))e^{A(x,q)}}{1-qx e^{A(x,q)}}.
$$
Plugging in~\eqref{eq:A-diffeq} yields
\begin{equation}\label{eq:B_diffeq}
f'(x) - (1+q+2q x A_x(x,q)) A_x(x,q) f(x) = x (1+q+2q x A_x(x,q)) A_x(x,q)^2.
\end{equation}
By implicit differentiation of~\eqref{eq:impl_eq_A}, we obtain
$$x A_x(x,q) = \frac{e^{A(x,q)} - e^{q A(x,q)}}{e^{qA(x,q)} - q e^{A(x,q)}}.$$
So the integrating factor of~\eqref{eq:B_diffeq} is
\begin{multline*}
\exp \bigg({- \int (1+q+2q x A_x(x,q)) A_x(x,q) \,dx }\bigg) \\
= \exp \bigg({- \int \Big(1+q + 2q \cdot \frac{e^{A(x,q)} - e^{q A(x,q)}}{e^{qA(x,q)} - q e^{A(x,q)}} \Big) A_x(x,q)\,dx }\bigg),
\end{multline*}
and the simple substitution $t = A(x,q)$ allows us to integrate explicitly, which reveals that $e^{-(1+q)A(x,q)} (e^{qA(x,q)} - q e^{A(x,q)} )^2$ is an integrating factor for the differential equation~\eqref{eq:B_diffeq}. In the same way, one can integrate the right side of~\eqref{eq:B_diffeq} (after multiplying by the integrating factor). Finally, we use the initial value $f(0) = A(0,q) = 0$ and solve for $f(x)$ to obtain
$$B(x,q) = f(x) = \frac{e^{2qA(x,q)} + ((1-q)^2A(x,q) - q - 1) e^{(1+q)A(x,q)} + q e^{2A(x,q)}}{(e^{qA(x,q)} - q e^{A(x,q)} )^2}.$$
Now we can use the asymptotic expansion of $A(x,q)$ in Lemma~\ref{lem:analytic_properties_A} to find that
$$B(x,q) = \frac{\log q}{2(q-1) (1-x/x_0)} + O \big( |x-x_0|^{-1/2} \big),$$
so singularity analysis yields
$$\frac{[x^n] B(x,q)}{[x^n] A(x,q)} \sim \frac{\log q}{q-1} \cdot \sqrt{\frac{\pi q}{2}} \cdot n^{3/2}$$
as $n \to \infty$. Since this quotient is precisely the mean path length of descent-biased trees $\cT_n^{(q)}$, we can summarize with the following proposition.
\begin{proposition}\label{prop:pathlength}
For fixed $q > 0$, the expected value of the path length of a descent-biased tree $\cT_n^{(q)}$ is asymptotically equal to
$$\frac{\log q}{q-1} \cdot \sqrt{\frac{\pi q}{2}} \cdot n^{3/2}$$
 as $n \to \infty$.
\end{proposition}

Again, for $q=1$, we obtain $\sqrt{\frac{\pi}{2}} n^{-3/2}$, which is the asymptotic mean path length of a Cayley tree.

\section{Local limits of descent-biased trees}
\label{sec:trees}

We now combine our results on biased permutations as well as those from the previous section to prove a local limit for our random tree model.

\subsection{Local limits of trees}

Let us recall the notion of local limit of a sequence of (random) locally finite rooted trees $(\cT_n)_{n \geq 1}$. We say that $\cT_*$ is the local limit of $(\cT_n)_{n \geq 1}$ if, for all fixed $r \geq 0$:
\begin{align*}
B_r\left( \cT_n \right) \underset{n \to \infty}{\overset{(d)}{\to}} B_r(\cT_*),
\end{align*}
where $B_r(T)$ denotes the closed ball of radius $r$ around the root of $T$ with respect to the graph distance. 

Local limits of trees conditioned by their size or other parameters have mainly been investigated in the case of Bienaymé--Galton--Watson trees  \cite{AD14a, AD14b, ABD17} or slightly more general models of trees \cite{Pag17, Stu22}. When the offspring distribution of the Bienaymé--Galton--Watson tree $\mu$ is critical (that is, of mean $1$), the local limit $\tau^*$ is called Kesten's tree, see~\cite{Kes86}. This tree is made of an infinite spine to which i.i.d.~$\mu$-Bienaymé--Galton--Watson trees are grafted. This applies in particular to our setting when $q=1$. In this case, the size-conditioned tree $\cT_n$ is a Bienaymé--Galton--Watson tree with offspring distribution $\Poi(1)$.
On the other hand, when $q=0$ in our model, the degree of the root of $\cT_n$ goes a.s. to $\infty$ as $n \to \infty$, so that the tree does not have a local limit in the sense defined above.

The difference between our case and the Bienaymé--Galton--Watson case is that there is more dependence between the different degrees in the tree, and crucial independence arguments from \cite{AD14a, AD14b, ABD17} cannot be used. However, the structure of the local limit $\cT_*$ of $\cT^{(q)}_n$ when $q \in (0,1]$ is fixed is quite similar to that of Kesten's tree, with an infinite spine to which trees are grafted that also follow a descent-biased distribution. These grafted trees are not independent anymore, and depend on each other via the structure of the infinite spine, which is itself, as we will show, locally distributed as a descent-biased permutation.

The idea of the proof of Theorem \ref{thm:locallimittree} is the following. We decompose the ancestral line $L(n)$ of the vertex labelled $n$ in $\cT_n^{(q)}$: conditionally on its height $h_n$, we may construct $L(n)$ from a random forest of $h_n$ trees, which we graft on a spine in such a way that their roots are ordered according to a descent-biased permutation of size $h_n$. Thus, we use Theorem \ref{thm:mainconstant} which states the existence of a local limit for such a permutation, to obtain Theorem \ref{thm:locallimittree}.

\subsection{Estimates concerning the vertex labelled $n$}\label{subsec:label_n}

We start with some results on the structure of the tree $\cT_n^{(q)}$. The first lemma investigates the height of the vertex labelled $n$, as well as the height of the whole tree $\cT_n^{(q)}$ and the size of the subtree on top of the vertex labelled $n$. 

\begin{lemma} 
\label{lem:heightn}
Let $q \in (0,1]$ be fixed. The following statements hold as $n \to \infty$:
\begin{itemize}
\item[(i)] Let $h_n$ be the height of vertex $n$ in the tree $\cT_n^{(q)}$. Then, for any positive integer $r$, 
\begin{align*}
\P(h_n \leq r) = o(1).
\end{align*}
\item[(ii)]
Let $H(\cT_n^{(q)})$ be the height of the tree $\cT_n^{(q)}$. Then, for every function $h(n)$ such that $h(n) n^{-3/4} \to \infty$, 
\begin{align*}
\P\big(H(\cT_n^{(q)}) \geq h(n)\big) = o(1).
\end{align*}

\item[(iii)]
For any $f_n=o(n^{1/4})$, the probability that the subtree rooted at the vertex labelled $n$ has size at least $n/f_n$ is $o(1)$.
\end{itemize}
\end{lemma}

\begin{proof}
For part (i), note that the probability for the height of vertex $i$ to be less than or equal to $r$ is non-increasing in $i$. So if the probability did not tend to $0$ (i.e., if its liminf was some $\epsilon > 0$), the expected number of vertices whose distance from the root is at most $r$ would have to be at least $\epsilon n$ for arbitrarily large values of $n$, contradicting Proposition~\ref{prop-r-radius-bounded}.

For part (ii) observe that if the height of a tree is $h$, then the path length must clearly be at least $1 + 2 + \cdots + h = \frac{h(h+1)}{2}$, since there is at least one vertex at distance $j$ from the root for every $j \leq h$. Thus the probability in question is bounded above by the probability that $\PL(\cT_n^{(q)}) \geq \frac{ h(n) (h(n)+1)}{2}$. The statement now follows immediately from Proposition~\ref{prop:pathlength} by means of the Markov inequality.

Let us finally prove (iii). Observe that deterministically, for any $K>0$, the number of vertices in an $n$-vertex tree $T$ that have at least $K$ descendants is at most $H(T) n/K$, where $H(T)$ is the height of the tree. Indeed, let $A_K(T)$ be the set of such vertices, and $M_K(T)$ the subset of $A_K(T)$ consisting of vertices such that none of their children is in $A_K(T)$. Then, the subtrees rooted at the elements of $M_K(T)$ are disjoint, and thus $|M_K(T)| \leq n/K$. In addition, any element of $A_K(T)$ is necessarily an ancestor of an element of $M_K(T)$, so that $|A_K(T)| \leq H(T) \, |M_K(T)| \leq H(T) n/K$.

Let $g_n$ be a function such that $f_n \ll g_n \ll n^{1/4}$. Letting $p_k$ be the probability that the vertex labelled $k$ in $\cT_n^{(q)}$ has a subtree of size $\geq n/f_n$ rooted at it, it is clear that $p_k$ is non-increasing in $k$. Hence, $$p_n \leq \P\left( H\left(\cT_n^{(q)}\right) \geq n/g_n \right) + \frac{1}{n} \E\left[ \big|A_{n/f_n} \big(\cT_n^{(q)}\big)\big| \,\Big|\, H\left(\cT_n^{(q)}\right) \leq n/g_n \right].$$

By part (ii), we have $\P\left( H(\cT_n^{(q)}) \geq n/g_n \right) = o(1)$. We also get from the previous deterministic argument that
\begin{align*}
\frac{1}{n} \E\left[ \big|A_{n/f_n} \big(\cT_n^{(q)}\big)\big| \,\Big|\, H\left(\cT_n^{(q)}\right) \leq n/g_n \right] \leq \frac{1}{n} \frac{n}{g_n} \frac{n}{(n/f_n)} = \frac{f_n}{g_n} = o(1). 
\end{align*}

This proves (iii).

\end{proof}

\subsection{Ancestral line of $n$}

Let us now consider the distribution of the ancestral line of vertex $n$. The proof of the following proposition, which provides an algorithm to construct a descent-biased tree of size $n$ with given height $h_n$ of the vertex labelled $n$, is quite straightforward.

\begin{proposition}
\label{prop:treeisforest}
Let $n \geq 1$ and $h \geq 0$. Denote by $F_{n,h} := (T_1, \ldots, T_{h+1})$ the random forest of $(h+1)$ trees of total size $n$, vertices labelled from $1$ to $n$, such that the root $\rho(T_{h+1})$ of $T_{h+1}$ has label $n$ and 
\begin{align*}
\P \left( F_{n,h}=(t_1, \ldots, t_{h+1}) \right) \propto \mathds{1}_{\sum_{i=1}^{h+1} |t_i|=n} \mathds{1}_{\ell(\rho(T_{h+1}))=n} q^{\sum_{i=1}^{h+1} d(t_i)},
\end{align*}
for any forest $(t_1, \ldots, t_{h+1})$ of $h+1$ trees of size at least $1$, with $n$ vertices labelled from $1$ to $n$.

Furthermore, let $\sigma$ be a $q$-descent-biased permutation of $\kS_{h}$ independent of $F_{n,h}$ and let $\rho_1, \ldots, \rho_{h}$ be the roots of $T_1, \ldots, T_h$ sorted in increasing order of labels. Construct now a tree $\tau_n$ from $F_{n,h}$ by adding all edges of the form $\{ \rho_{\sigma(i)}, \rho_{\sigma(i+1)}\}$, $1 \leq i \leq h$, and the edge $\{ \rho_{\sigma(h)}, n \}$, and root it at $\rho_{\sigma(1)}$. Then, $\tau_n$ has the same distribution as $\cT_{n}^{(q)}$, conditioned on the height of vertex $n$ being $h$:
\begin{align*}
\left(\cT_{n}^{(q)} \, | \, h_n=h \right) \overset{(d)}{=} \tau_n.
\end{align*}
\end{proposition}

\subsection{Forest of descent-biased trees}

By Proposition \ref{prop:treeisforest}, studying the ancestral line of $n$ boils down to studying a random descent-biased forest. In what follows, let $m$ be a function from the set of positive integers to itself such that $m(n) \to +\infty$ as $n \to +\infty$ and $m(n)=o(n)$. Let $F_n$ be a forest of $m(n)$ trees with vertices labelled by $\{1, \ldots, n \}$, such that
\begin{align*}
\P\left(F_n = (t_1, \ldots, t_{m(n)})\right) \propto \mathds{1}_{\sum |t_j|=n} q^{\sum d(t_j)}.
\end{align*}

For any $k \geq 1$ and any tree $T \in \kT_k$, denote by $s(T)$ the shape of $T$, that is, the rooted unlabelled tree obtained by removing all labels of the vertices of $T$. The next proposition is the main tool in the proof of Theorem \ref{thm:locallimittree}. 

\begin{proposition}
\label{prop:shapesoftrees}
\begin{itemize}
\item[(i)]
Fix a rooted unlabelled tree $\tau$. For any $0 \leq a < b \leq 1$, let $N_{\tau}(a,b)$ be the number of trees in $F_n$ of shape $\tau$ whose root label belongs to $( na,nb] \cap \Z$. Then, there exists a continuous function $G_\tau: [0,1] \to \R_+$ such that, for all $0 \leq a < b \leq 1$:

\begin{align*}
\frac{N_{\tau}(a,b)}{m(n)} \underset{n \to \infty}{\overset{\P}{\to}} \int_a^b G_\tau(s) ds.
\end{align*}

In what follows, set $G(s) := \sum_{\tau} G_\tau(s)$, the sum being taken over all possible finite shapes $\tau$, and, for all $x \in [0,1]$, define
\begin{align*}
L(x) := \sup \left\{ a \in [0,1], \int_0^a G(s) ds \leq x \right\}. 
\end{align*}

\item[(ii)] 
We have $\int_0^{L(x)} G(s) ds = x$.

\item[(iii)]
With the same notation, let $\ell(\rho_1), \ldots, \ell(\rho_{m(n)})$ be the labels of the roots of the $m(n)$ trees ordered increasingly. Then, the following convergence holds in probability, with respect to $|| \cdot ||_\infty$ on $[0,1]$:
\begin{align*}
\frac{1}{n} \left(\ell(\rho_{\lfloor x \, m(n) \rfloor})\right)_{0 \leq x \leq 1} \underset{n \to \infty}{\overset{\P}{\to}} \left(L(x)\right)_{0 \leq x \leq 1}.
\end{align*}

\item[(iv)]
Fix a rooted unlabelled tree $\tau$. With the same notation as in (iii), for any $0 \leq u < v \leq 1$, let $R_{\tau}(u,v)$ be the number of trees in $F_n$ of shape $\tau$ with root label in the set $\{ \ell(\rho_{\lfloor u \, m(n)\rfloor +1}), \ldots, \ell(\rho_{\lfloor v \, m(n) \rfloor}) \}$ (that is, whose root label is between the $(u m(n))$-th smallest and the $(vm(n))$-th smallest among the root labels). Then
\begin{align*}
\frac{R_{\tau}(u,v)}{m(n)} \underset{n \to \infty}{\overset{\P}{\to}} \int_{L(u)}^{L(v)} G_\tau(s) ds.
\end{align*}
\end{itemize}

In addition, for any $M_1(n), M_2(n)$ such that $1 \ll M_1(n) \ll M_2(n) \ll n$, these estimates hold uniformly in $m(n)$ as $M_1(n) \leq m(n) \leq M_2(n)$.
\end{proposition}

In other words, for all $a \in [0,1]$, the proportion of trees with shape $\tau$ whose roots have labels among the smallest $a m(n)$ (among the root labels in $F_n$) is asymptotically deterministic.
The proof of this proposition is postponed to the end of the section. Let us first prove that it implies Theorem \ref{thm:locallimittree}.

To this end, we actually show something stronger. Define the $r$-hull of a rooted tree $T$ with $n$ labelled vertices as follows: let $a^n_r$ be the $r$-th lowest ancestor of the vertex labelled $n$ (or $n$ itself if the height $h_n$ of $n$ is less than $r$). In particular, $a^n_r$ has height $r-1$ for all $r \leq h_n+1$. Then the $r$-hull $T_{/r}$ is obtained from $T$ by removing all descendants of $a^n_r$.
We also let $F_{/r}$ denote the ordered forest of $(r+1)$ trees obtained by removing the edges $\{ \{ a^n_{i}, a^n_{i+1} \}, 1 \leq i \leq r \}$. We prove that, for any fixed $r$, the (unlabelled) $r$-hull of $\cT_n^{(q)}$ converges in distribution. This in particular implies directly the one-endedness of the local limit.

\begin{proof}[Proof of Theorem \ref{thm:locallimittree}]
First, observe that, by Lemma \ref{lem:heightn} (iii), with high probability the size $X_n$ of the subtree rooted at $n$ is $\leq n^{4/5}$. Furthermore, conditionally on this size, this subtree is independent of the rest of the tree. Hence, up to relabelling the vertices in increasing order, by Proposition \ref{prop:treeisforest}, the connected components of $\cT_n^{(q)} \setminus \{\{a^n_i, a^n_{i+1} \}, 1 \leq i \leq h_n \}$, without the component containing vertex $n$, behave like a descent-biased forest on $n-X_n$ vertices. Hence, we can apply Proposition \ref{prop:shapesoftrees} to the forest $F_{/r}$.

Fix $r \geq 1$, and consider an $r$-tuple 
$$A_r := (\tau_1, \ldots, \tau_r)$$ 
such that, for all $1 \leq i \leq r$, $\tau_i$ is a finite rooted unlabelled tree. Let again $a^n_1, \ldots, a^n_{r+1}$ be the first $r+1$ ancestors of the vertex labelled $n$ in $\cT_n^{(q)}$ (with high probability, by Lemma \ref{lem:heightn} (i), $r+1 < h_n$). In particular, $a^n_1$ is the root of $\cT_n^{(q)}$. For all $1 \leq i \leq r$, let $T_i$ be the connected component of $\cT_n^{(q)} \setminus \{ \{a^n_{i}, a^n_{i+1} \}, 1 \leq i \leq r\}$ containing $a^n_i$. Finally, let $p_n(A_r)$ be the probability that $(s(T_1), s(T_2), \ldots, s(T_r)) = A_r$. We want to prove that $p_n(A_r)$ converges as $n \to \infty$. More precisely, let $D^{(r)}: [0,1]^{r} \to \R_+$ be the asymptotic density of $\frac{1}{m} \left( S_{m}^{(q)}(1), \ldots, S_m^{(q)}(r)\right)$ as $m \to \infty$ (which exists by Theorem \ref{thm:mainconstant}). Then, we claim that, with $L$ and $G$ as in Proposition~\ref{prop:shapesoftrees},
\begin{equation}
\label{eq:convergencepnAr}
p_n(A_r) \underset{n \to \infty}{\to} p_\infty(A_r) := \int_{[0,1]^{r}} D^{(r)}(s_1, \ldots, s_r) \prod_{i=1}^r G_{\tau_i}(L(s_i)) L'(s_i) ds_i.
\end{equation}

On the other hand, observe that we have
\begin{align*}
\sum_{\tau_{r+1}} p_\infty(\tau_1, \ldots, \tau_{r+1}) &= \int_{[0,1]^{r+1}} D^{(r+1)}(s_1, \ldots, s_r, s) \prod_{i=1}^r G_{\tau_i}(L(s_i)) L'(s_i) ds_i \\
&\qquad \sum_{\tau_{r+1}} G_{\tau_{r+1}}(L(s)) L'(s) ds\\
&= \int_{[0,1]^{r}} \prod_{i=1}^r G_{\tau_i}(L(s_i)) L'(s_i) ds_i \int_{s=0}^1 D^{(r+1)}(s_1, \ldots, s_r, s) G(L(s)) L'(s) ds.
\end{align*}
Since, by Proposition \ref{prop:shapesoftrees} (ii), $\int_0^{L(s)} G(u) du = s$ for all $s \in [0,1]$, we have by differentiation $G(L(s)) L'(s) = 1$. Hence, 
\begin{align*}
\sum_{\tau_{r+1}} p_\infty(\tau_1, \ldots, \tau_{r+1}) = \int_{[0,1]^{r+1}} \prod_{i=1}^r G_{\tau_i}(L(s_i)) L'(s_i) ds_i D^{(r)}(s_1, \ldots, s_r) ds = p_\infty(\tau_1, \ldots, \tau_r).
\end{align*}

In other words, with high probability the size of the hull of radius $r$ is stochastically bounded. This along with \eqref{eq:convergencepnAr} proves Theorem \ref{thm:locallimittree}.

We now only need to prove \eqref{eq:convergencepnAr}.
To this end, let $\rho_1, \ldots, \rho_{h_n}$ be the roots of $\{T_1, \ldots, T_{h_n}\}$ sorted in increasing label order, so that the trees are sorted in increasing order of their root labels. Then we have, by Proposition~\ref{prop:treeisforest}:
\begin{equation}
\label{eq:pnarsum}
p_n(A_r) = \sum_{1 \leq \ell_1, \ldots, \ell_r \leq h_n} \P \left(S_{h_n}^{(q)}(1)=\ell_1, \ldots, S_{h_n}^{(q)}(r)=\ell_r \right) \P \left( s\left(T_{\ell_1}\right) = \tau_1, \ldots, s\left(T_{\ell_r}\right) = \tau_r \right).
\end{equation}

We now define a model of balls-in-boxes, which turns out to be useful in the proof of \eqref{eq:convergencepnAr}. We fix an integer $K \geq 1$, and set, for all $0 \leq i \leq K$, $a_i = \lfloor i h_n/K \rfloor$. In particular $a_0=0$ and $a_K=h_n$. Set also $E_i=( a_{i-1}, a_i ] \cap \Z$ for all $1 \leq i \leq K$ (one should think of it as $K$ boxes, where $K$ will be taken large but fixed as $n \to \infty$). We denote by $Z_{r}$ the set of $r$-tuples $(i_1, \ldots, i_{r})$ such that $i_j \neq i_k$ for any $j \neq k$. This corresponds to choosing $r$ ordered boxes among the $K$.

Fix now an ordered element $E := (i_1, \ldots, i_{r}) \in Z_{r}$, and denote by $Y_E$ the subset of $\llbracket 1, h_n \rrbracket^{r}$ made of the $r$-tuples $(\ell_1, \ldots, \ell_{r})$ such that $\ell_j \in E_{i_j}$ for all $1 \leq j \leq r$. In words, considering $h_n$ balls, this corresponds to allocations of $r$ of these balls to the boxes of $E$. Denote also by $Y_{r}$ the subset of $\llbracket 1, h_n \rrbracket^{r}$ made of $r$-tuples $(\ell_1, \ldots, \ell_{r})$ such that no two of these elements belong to the same $E_i$ (that is, allocations of $r$ balls to $r$ different boxes), and denote by $\overline{Y}_{r}$ its complement in $\llbracket 1, h_n \rrbracket^{r}$. Finally, for $E \in Y_{r}$, let also $c_E, C_E$ be the minimum (resp.~maximum) of $\P \left(S_{h_n}^{(q)}(1)=\ell_1, \ldots, S_{h_n}^{(q)}(r)=\ell_{r} \right)$ for $(\ell_1, \ldots, \ell_{r}) \in E$.

Then we claim the following:
\begin{itemize}
\item
\begin{equation}
\label{eq:eq1}
h_n^{-r} \sum_{(\ell_1, \ldots, \ell_{r}) \in Y_E} \P \left( s\left(T_{\ell_1}\right) = \tau_1, \ldots, s\left(T_{\ell_r}\right) = \tau_r \right) \underset{n \to \infty}{\to} \prod_{j=1}^{r} \int_{L(a_{i_j-1})}^{L(a_{i_j})} G_{\tau_{j}}(s) ds;
\end{equation}
\item
\begin{equation}
\label{eq:eq2}
\lim_{K \to \infty} \limsup_{n \to \infty} \sum_{(\ell_1, \ldots, \ell_r) \in \overline{Y}_{r}} \P \left(S_{h_n}^{(q)}(1)=\ell_1, \ldots, S_{h_n}^{(q)}(r)=\ell_r \right) = 0.
\end{equation}
\item We have
\begin{equation}
\label{eq:eq3}
\lim_{K \to \infty} \limsup_{n \to \infty}\sup_{E \in Y_{r}} h_n^{r} (C_E-c_E) = 0,
\end{equation}
and 
\begin{equation}
\label{eq:eq4}
\lim_{K \to \infty} \limsup_{n \to \infty}\sup_{E \in Y_{r}} h_n^{r} C_E \in (0,\infty).
\end{equation}
\end{itemize}

In words, \eqref{eq:eq2} states that, in the computation of $p_n(A_r)$, we can restrict ourselves to the $r$-tuples such that no two $\ell_i$'s are in the same box. In addition, \eqref{eq:eq3} and \eqref{eq:eq4} state that $\P \left(S_{h_n}^{(q)}(1)=\ell_1, \ldots, S_{h_n}^{(q)}(r)=\ell_{r} \right)$ asymptotically only depends on the boxes in which $\ell_1, \ldots, \ell_{r}$ are, and that the second term in \eqref{eq:pnarsum} is asymptotically a constant times $h_n^{-r}$. Finally, \eqref{eq:eq1} shows the convergence of our quantity of interest.

It is clear that \eqref{eq:eq1}, \eqref{eq:eq2}, \eqref{eq:eq3} and \eqref{eq:eq4} imply \eqref{eq:convergencepnAr} and therefore Theorem \ref{thm:locallimittree}.

To prove \eqref{eq:eq1}, observe that

\begin{align*}
h_n^{-r} \sum_{(\ell_1, \ldots, \ell_{r}) \in Y_E} &\P \left( s\left(T_{\ell_1}\right) = \tau_1, \ldots, s\left(T_{\ell_r}\right) = \tau_r \right) \\
&= h_n^{-r} \sum_{(\ell_1, \ldots, \ell_{r}) \in Y_E} \E\left[\mathds{1}_{s\left(T_{\ell_1}\right) = \tau_1} \cdots \mathds{1}_{s\left(T_{\ell_r}\right) = \tau_r} \right]\\
&= h_n^{-r} \E\left[ R_{\tau_1}\left( a_{i_1-1}, a_{i_1} \right) \cdots R_{\tau_r}\left( a_{i_{r}-1}, a_{i_{r}} \right) \right].
\end{align*}

Since, by Proposition \ref{prop:shapesoftrees} (iv), all variables $h_n^{-1} R_{\tau_{j}}(a_{i_j-1}, a_{i_j})$ are a.s.~bounded and converge in probability respectively to $\int_{L(a_{i_j-1})}^{L(a_{i_j})} G_{\tau_{j}}(s) ds$, we get \eqref{eq:eq1}.

In order to prove \eqref{eq:eq2}, recall that the variable $(X_1, \ldots, X_{r})$ of Theorem \ref{thm:mainconstant} is absolutely continuous with respect to the Lebesgue measure on $[0,1]^{r}$. Therefore, there exists a constant $C>0$ such that, for all $n, K$:
\begin{align*}
\sum_{(\ell_1, \ldots, \ell_r) \in \overline{Y}_{r}} \P \left(S_{h_n}^{(q)}(1)=\ell_1, \ldots, S_{h_n}^{(q)}(r)=\ell_r \right) \leq C \P \left( \exists j\neq k, d(U_j, U_k) \leq 1/K \right),
\end{align*}
where $(U_j, 1 \leq j \leq r)$ are i.i.d.~uniform on $[0,1]$. In particular, this goes to $0$ as $K \to \infty$.

Finally, \eqref{eq:eq3} and \eqref{eq:eq4} are consequences of Proposition \ref{prop:loclim1} (ii).
\end{proof}

\subsection{Proof of Proposition \ref{prop:shapesoftrees}}

It turns out that Proposition \ref{prop:shapesoftrees} (iii) and (iv) are consequences of Proposition \ref{prop:shapesoftrees} (i). Hence, we first prove (i) and (ii), and then explain how to obtain (iii) and (iv) from (i). We emphasize the fact that all estimates that we prove are uniform for any $m(n) \in [M_1(n), M_2(n)] \cap \Z$, where $M_1, M_2$ satisfy $1 \ll M_1(n) \ll M_2(n) \ll n$. This is a consequence of the uniformity in Proposition \ref{prop:power_coefficients}.

We need the following lemma on the number of trees of given shape in $F_n$. For any unlabelled rooted tree $\tau$, let $N_\tau := N_\tau(0,1)$ denote the number of trees in $F_n$ that have shape $\tau$. We also write
\begin{align*}
P_\tau := \P \left( s\left(\cT_{|\tau|}^{(q)}\right) = \tau \right)
\end{align*}
for the probability that a $q$-descent-biased tree of size $|\tau|$ has shape $\tau$. Recall that $x_0 =q^{\frac{q}{1-q}}$ is the radius of convergence of the generating function $A(x,q)$, and recall also the notation $C_q = \frac{\log(1/q)}{1-q} = A(x_0,q)$.

\begin{lemma}
\label{lem:treesofshapetau}
For any fixed unlabelled rooted tree $\tau$, we have
\begin{align*}
\frac{N_\tau}{m(n)} \overset{\P}{\underset{n \to \infty}{\to}} [x^{|\tau|}] A(x,q) \frac{x_0^{|\tau|}}{C_q} P_\tau.
\end{align*}
\end{lemma}

\begin{proof}[Proof of Lemma \ref{lem:treesofshapetau}]
Define, for all $m \geq 1$,
\begin{align*}
F_m(x,q) := \sum_{n \geq 1, d \geq 0} \frac{f_m(n,d)}{n!} x^n q^d,
\end{align*}
where $f_m(n,d)$ is the number of forests of $m$ trees with $n$ vertices and $d$ descents. It is clear that $F_m(x,q) = (A(x,q))^m$. 
Fix $\tau$ and $m \geq 1$, and let 
\begin{align*}
\tilde{F}_m(x,q) := \sum_{n \geq 1, d \geq 0} \frac{\tilde{f}_m(n,d)}{n!} x^n q^d,
\end{align*}
where $\tilde{f}_m(n,d)$ is the number of forests of $m$ trees with $n$ vertices and $d$ descents, such that the first tree has shape $\tau$. Then, 
\begin{align*}
[x^n] \tilde F_m(x,q) = [x^{|\tau|}] A(x,q) P_\tau [x^{n-|\tau|}] A(x,q)^{m-1}.
\end{align*}
Hence, we get
\begin{align*}
\P\left( s(t_1)=\tau \right) = \frac{[x^n] \tilde{F}_m(x,q)}{[x^n] F_m(x,q)} = \frac{[x^{|\tau|}] A(x,q) [x^{n-|\tau|}] A^{m-1}(x,q)}{[x^n]A^m(x,q)} P_\tau.
\end{align*}
Observe now that, uniformly as $n,m \to \infty$ and $n/m \to \infty$, it follows from Proposition \ref{prop:power_coefficients} that
\begin{align*}
\frac{[x^{n-|\tau|}]A^{m-1}(x,q)}{[x^n]A^m(x,q)} \sim \frac{1}{C_q} x_0^{|\tau|},
\end{align*}
and thus that
\begin{equation}
\label{eq:cvofsizet1}
\P\left(s(t_1)=\tau\right) \underset{n \to \infty}{\to} [x^{|\tau|}] A(x,q) \frac{x_0^{|\tau|}}{C_q} P_\tau.
\end{equation}
Now observe that, for all fixed shapes $\tau_1, \tau_2$, as $n,m \to \infty$ and $n/m \to \infty$,
\begin{align}
\label{eq:cvofsizescouple}
\P\left( s(t_1)=\tau_1, s(t_2)=\tau_2 \right) &= \frac{[x^{|\tau_1|}] A(x,q) P_{\tau_1} [x^{|\tau_2|}] A(x,q) P_{\tau_2} [x^{n-|\tau_1|-|\tau_2|}] A^{m-2}(x,q)}{[x^n]A^m(x,q)}\\
&\sim P_{\tau_1} P_{\tau_2} \frac{[x^{|\tau_1|}] A(x,q) [x^{|\tau_2|}] A(x,q) x_0^{|\tau_1|+|\tau_2|}}{C_q^2}.
\end{align}
Lemma \ref{lem:treesofshapetau} follows from a classical second moment argument, since \eqref{eq:cvofsizet1} and \eqref{eq:cvofsizescouple} imply that the variance of $N_\tau/m(n)$ goes to $0$ as $n \to \infty$.

\end{proof}

We can now prove Proposition \ref{prop:shapesoftrees}.

\begin{proof}[Proof of Proposition \ref{prop:shapesoftrees} (i)]
Fix an unlabelled rooted tree $\tau$ of size $|\tau|$. Then, by Lemma \ref{lem:treesofshapetau},
\begin{align*}
\frac{N_\tau(0,1)}{m(n)} \overset{\P}{\underset{n \to \infty}{\to}} [x^{|\tau|}] A(x,q) \frac{x_0^{|\tau|}}{C_q} \, P_\tau.
\end{align*}
Now, for all $1 \leq i \leq |\tau|$, let
\begin{align*}
p_i(\tau) = \P \left( \ell(\rho)=i \left| s\left(\cT_{|\tau|}^{(q)}\right) \right. = \tau \right)
\end{align*}
be the probability that the root $\rho$ of $\cT_{|\tau|}^{(q)}$ is labelled $i$ knowing that $\cT_{|\tau|}^{(q)}$ has shape $\tau$. 

Observe that, conditionally on their sizes, the labels of the vertices of the trees of shape $\tau$ form independent subsets of $\llbracket 1, n \rrbracket$ of size $|\tau|$ (conditionally on being disjoint). We obtain immediately, jointly for all $a,b$:
\begin{align*}
\frac{N_\tau(a,b)}{m(n)} \overset{\P}{\underset{n \to \infty}{\to}} [x^{|\tau|}] A(x,q) \frac{x_0^{|\tau|}}{C_q} \, P_\tau \, \sum_{i=1}^{|\tau|} p_i(\tau) \int_a^b u_i(s) ds,
\end{align*}
where $u_i$ is the density of the $i$-th order statistic of $|\tau|$ i.i.d.~uniform variables on $[0,1]$. The result follows, with
\begin{align*}
G_\tau: s \mapsto [x^{|\tau|}] A(x,q) \frac{x_0^{|\tau|}}{C_q} \, P_\tau \, \sum_{i=1}^{|\tau|} p_i(\tau) u_i(s).
\end{align*}
Since all $u_i$'s are continuous, $G_\tau$ is continuous on $[0,1]$.
\end{proof}

\begin{proof}[Proof of Proposition \ref{prop:shapesoftrees} (ii), (iii) and (iv)]
Since $G$ is (strictly) positive, it is immediate that $\int_0^{L(x)} G(s) ds = x$. Let us now prove (iii). Since $x \mapsto \ell(\rho_{\lfloor x \, h_n \rfloor})$ is nondecreasing, by Dini's theorem, we only need to check that $L$ is continuous, and that the convergence holds pointwise in probability. The continuity of $L$ is clear, being the inverse of an increasing continuous function. To prove the pointwise convergence, observe that we obtain from Proposition \ref{prop:shapesoftrees} (i) that, for all $a \in [0,1]$,
\begin{align*}
\P\left( \ell(\rho_{\lfloor x h_n \rfloor}) > n a \right) = \P\left( \sum_\tau N_\tau(0,a) < \lfloor x h_n \rfloor \right) \underset{n \to \infty}{\to} \P\left( \int_0^a G(s) ds \leq x \right) = \mathds{1}_{\int_0^a G(s) ds \leq x}.
\end{align*}
This implies the result.

Finally, (iv) is just a consequence of (i) and (iii).
\end{proof}

\section*{Acknowledgments}

The first author was supported by the Austrian Science Fund (FWF) under grant P33083. The second author was supported by the Knut and Alice Wallenberg Foundation (KAW 2017.0112) and the Swedish research council (VR), grant 2022-04030.

\bibliographystyle{abbrv}
\bibliography{BibTreeInv}

\end{document}